\documentclass[12pt]{article}
\usepackage{ct}

\usepackage{enumerate}
\usepackage{thmtools}
\usepackage{xcolor}
\usepackage{tikz}
\usetikzlibrary{calc, shapes}

\crefname{observation}{Observation}{Observations}

\specs{17}{3 (1)}{2023}{}

\dateline{Nov 10, 2021}{Jan 19, 2023}{Mar 15, 2023}



\title{An exact characterization of saturation \\ for permutation matrices}

\author[1]{Benjamin Aram Berendsohn\thanks{Supported by DFG grant KO 6140/1-1.}}

\affil[1]{%
Freie Universit\"at Berlin, Institut f\"ur Informatik, Berlin, Germany

\email{benjamin.berendsohn@fu-berlin.de}%
}

\newcommand{\N}{\mathbb{N}}

\newcommand{\fO}{\mathcal{O}}
\newcommand{\fP}{\mathcal{P}}

\newcommand{\rh}{\mathrm{h}}
\newcommand{\rv}{\mathrm{v}}

\newcommand{\rL}{\mathrm{L}}
\newcommand{\rR}{\mathrm{R}}
\newcommand{\rT}{\mathrm{T}}
\newcommand{\rB}{\mathrm{B}}

\newcommand{\rM}{\mathrm{M}}
\newcommand{\rH}{\mathrm{H}}
\newcommand{\rV}{\mathrm{V}}

\newcommand{\ex}{\mathrm{ex}}
\newcommand{\sat}{\mathrm{sat}}

\newcommand{\rot}{\mathrm{rot}}
\newcommand{\rev}{\mathrm{rev}}
\newcommand{\trans}{\mathrm{trans}}

\newcommand{\vd}{\mathrm{d}^\rv}
\newcommand{\hd}{\mathrm{d}^\rh}
\newcommand{\width}{\mathrm{width}}
\newcommand{\height}{\mathrm{height}}

\newcommand{\ltv}{<_\rv}
\newcommand{\lth}{<_\rh}
\newcommand{\lev}{\le_\rv}
\newcommand{\leh}{\le_\rh}

\newcommand{\zeromat}{\mathbf{0}}
\newenvironment{smallbulletmatrix}{%
	\renewcommand{\o}{\bullet}% Overrides o with diacritic
	\renewcommand{\d}{\cdot}% Overrides diacritic
	\newcommand{\g}{\textcolor{green}{\o}}

	\left(\begin{smallmatrix}%
	}{%
	\end{smallmatrix}\right)%
}
\newcommand{\decmat}{\left(\begin{smallmatrix}A&\zeromat\\\zeromat&B\end{smallmatrix}\right)}
\newcommand{\rdecmat}{\left(\begin{smallmatrix}\zeromat&B\\A&\zeromat\end{smallmatrix}\right)}

\newenvironment{casedist}{%
	\begin{enumerate}[\hspace{2mm}\emph{Case} 1:]%
	}{%
	\end{enumerate}%
}

\newcounter{resumeEnum}

\tikzset{
	S/.style = {red},
	L/.style = {blue},
	R/.style = {green}
}

\newcommand{\tPoint}[1]{\node[circle,fill,inner sep=1.2pt] at (#1) {};}
\newcommand{\tNamedStyledPoint}[4]{\tPoint{#1}\node[#2] at (#1) {#3};}
\newcommand{\tNamedPoint}[3]{\tNamedStyledPoint{#1}{#2}{#3}{}}

\begin{document}
%\writedatatofile

\maketitle

% ABSTRACT
% CT papers must include an abstract. The abstract should consist of a
% succinct statement of background followed by a listing of the
% principal new results that are to be found in the paper. The abstract
% should be informative, clear, and as complete as possible. Phrases
% like "we investigate..." or "we study..." should be kept to a minimum
% in favor of "we prove that..."  or "we show that...".  Do not
% include equation numbers, unexpanded citations (such as "[23]"), or
% any other references to things in the paper that are not defined in
% the abstract. The abstract may be distributed without the rest of the
% paper so it must be entirely self-contained.  Try to include all words
% and phrases that someone might search for when looking for your paper.
% You can use some basic LaTeX commands in the abstract, but not any
% user defined macros. 

\begin{abstract}
	A 0-1 matrix $M$ \emph{contains} a 0-1 matrix \emph{pattern} $P$ if we can obtain $P$ from $M$ by deleting rows and/or columns and turning arbitrary 1-entries into 0s.
	The saturation function $\mathrm{sat}(P,n)$ for a 0-1 matrix pattern $P$ indicates the minimum number of 1s in an $n \times n$ 0-1 matrix that does not contain $P$, but changing any 0-entry into a 1-entry creates an occurrence of $P$. Fulek and Keszegh recently showed that each pattern has a saturation function either in $\mathcal{O}(1)$ or in $\Theta(n)$.
	We fully classify the saturation functions of \emph{permutation matrices}.
\end{abstract}

% TABLE OF CONTENTS, LIST OF FIGURES, LIST OF TABLES
% Please, do not include a table of contents, a list of figures, or a
% list of tables. They will be removed by the editors (and the command
% is actually redefined in the ct.sty file).

%%%%%%%%%%%%%%%%%%%%%%%%%%%%%%%%%%%%%%%%%%%%%%%%%%%
%%%%%%%%%%%%%%%%%%%%%%%%%%%%%%%%%%%%%%%%%%%%%%%%%%%
	
	\section{Introduction}
	
	In this paper, all matrices are  0-1 matrices. For cleaner presentation, we write matrices with dots~($\begin{smallmatrix}\bullet\end{smallmatrix}$) instead of 1s and spaces instead of 0s, for example:
	\begin{align*}
		\left(
		\begin{smallmatrix}
			0&1&0\\
			0&0&1\\
			1&0&0
		\end{smallmatrix}\right)
		=
		\begin{smallbulletmatrix}
			&\o&  \\
			&  &\o\\
			\o&  &  
		\end{smallbulletmatrix}
	\end{align*}
	In line with this notation, we call a row or column \emph{empty} if it only contains 0s. Furthermore, we refer to changing an entry from 0 to 1 as \emph{adding} a 1-entry, and to the reverse as \emph{removing} a 1-entry.

	A \emph{pattern} is a matrix that is not all-zero. A matrix $M$ \emph{contains} a pattern $P$ if we can obtain~$P$ from $M$ by deleting rows and/or columns, and removing arbitrary 1-entries. If $M$ does not contain $P$, we say $M$ \emph{avoids} $P$.
	Matrix pattern avoidance can be seen as a generalization of two other well-known areas in extremal combinatorics. Pattern avoidance in permutations (see, e.g., Vatter's survey \cite{Vatter2014}) corresponds to the case where both $M$ and $P$ are permutation matrices; and forbidden subgraphs in bipartite graphs correspond to avoiding a pattern $P$ and all other patterns obtained from $P$ by permutation of rows and/or columns.\footnote{For this, we interpret the $M$ and $P$ as adjacency matrices of bipartite graphs.} There are also close connections to the extremal theory of ordered graphs~\cite{PachTardos2006} and posets~\cite{GerbnerEtAl2022}.
	
	A classical question in extremal graph theory is to determine the maximum number of edges in an $n$-vertex graph avoiding a fixed pattern graph $H$. The corresponding problem in forbidden submatrix theory is determining the maximum \emph{weight} (number of 1s) of an $m \times n$ matrix avoiding the pattern $P$, denoted by $\ex(P, m, n)$. We call $\ex(P,n) = \ex(P,n,n)$ the \emph{extremal function} of the pattern $P$. The study of the extremal function originates in its applications to (computational) geometry \cite{Mitchell1987,Fueredi1990,BienstockGyoeri1991}. A systematic study initiated by Füredi and Hajnal \cite{FuerediHajnal1992} has produced numerous results (e.g., \cite{Klazar2000,Klazar2001,MarcusTardos2004,Tardos2005,Keszegh2009,Fulek2009,Geneson2009,Pettie2011,Pettie2011a}), and further applications in the analysis of algorithms have been discovered \cite{Pettie2010,ChalermsookEtAl2015}.
	
	A natural counterpart to the extremal problem is the \emph{saturation problem}. A matrix $M$ is \emph{saturating} for a pattern $P$, or \emph{$P$-saturating} if it avoids $P$ and is maximal in this respect, i.e., adding a 1-entry anywhere creates an occurrence of $P$. Clearly, $\ex(P,m,n)$ can also be defined as the maximum weight of an $m \times n$ matrix that is $P$-saturating. The function $\sat(P,m,n)$ indicates the \emph{minimum} weight of an $m \times n$ matrix that is $P$-saturating. We focus on square matrices and the \emph{saturation function} $\sat(P,n) = \sat(P,n,n)$.
	
	The saturation problem for matrix patterns was first considered by Brualdi and Cao \cite{BrualdiCao2021} as a counterpart of saturation problems in graph theory.\footnote{We refer to \cite{FulekKeszegh2021} for references to graph saturation results.} Fulek and Keszegh~\cite{FulekKeszegh2021} started a systematic study. They proved that, perhaps surprisingly, every pattern $P$ satisfies ${\sat(P,n) \in \fO(1)}$ or $\sat(P,n) \in \Theta(n)$, where the hidden constants depend on $P$. This is in stark contrast to the extremal problem, where a wide range of different orders of magnitude is attained by various patterns (from linear and quasi-linear~\cite{Keszegh2009,Pettie2011}, to nearly quadratic~\cite{AlonEtAl1999}). Fulek and Keszegh also present large classes of patterns with linear saturation functions. For our purposes, their most important result is that every \emph{decomposable} pattern has linear saturation function. We call a pattern $P$ decomposable if it has the form
	\begin{align*}
		\begin{pmatrix} A & \zeromat \\ \zeromat & B \end{pmatrix} \text{ or } \begin{pmatrix} \zeromat & A \\ B & \zeromat \end{pmatrix}
	\end{align*}
	for two matrices $A, B \neq \zeromat$, where $\zeromat$ denotes an all-0 matrix of the appropriate size. Otherwise, we call $P$ \emph{indecomposable}. Patterns of the first form $\decmat$ are called \emph{sum decomposable}, and patterns not of that form are called \emph{sum indecomposable}.\footnote{These terms are derived from the theory of permutation patterns (see, e.g., Vatter \cite{Vatter2014}). We are not aware of a standard term for this property in the context of 0-1 matrices.}
	
	\begin{figure}
		\begin{align*}
			Q = \begin{smallbulletmatrix}
				  &\o&  &  &  \\
				  &  &  &  &\o\\
				  &  &\o&  &  \\
				\o&  &  &  &  \\
				  &  &  &\o&  
			\end{smallbulletmatrix}
		\end{align*}
		\caption{The matrix with saturation function $\fO(1)$ found by Fulek and Keszegh~\cite{FulekKeszegh2021}.}\label{fig:mat-FK}
	\end{figure}
	
	Fulek and Keszegh also found a single non-trivial pattern with bounded saturation function ($Q$, pictured in \cref{fig:mat-FK}), and conjectured that there are many more. Geneson~\cite{Geneson2021} recently confirmed this by proving that almost all \emph{permutation matrices} have bounded saturation function. A permutation matrix is a matrix with exactly one 1-entry in each row and each column. A different class of matrices with bounded saturation function, containing both permutation matrices and non-permutation matrices were found recently by the author~\cite{Berendsohn2020}.\footnote{These results have been incorporated into this paper in \cref{sec:witnesses,sec:4-trav}.}
	
	In this paper, we show that, in fact, \emph{all} indecomposable permutation matrices have bounded saturation function. This completes the characterization of permutation matrices in terms of their saturation function.
	
	\begin{theorem}\label{p:main}
		A permutation matrix has linear saturation function if and only if it is decomposable.
	\end{theorem}
	
	A simple generalization of the technique that Fulek and Keszegh used to prove that $\sat(Q,n)$ is bounded implies the following: To prove \cref{p:main}, it is sufficient to find a \emph{vertical witness} for every indecomposable permutation matrix $P$, where we define a vertical witness for $P$ to be a matrix $M$ (of arbitrary size) that avoids $P$, has an empty row, and adding a 1-entry in that empty row creates an occurrence of $P$ in $M$.
	
	We therefore construct vertical witnesses for all permutation matrices. Our constructions are based on the fact that every indecomposable permutation matrix contains a \emph{spanning oscillation} (defined in \cref{sec:spanosc}).
	
	We also generalize a partial result to a class that contains non-permutation patterns:
	
	\begin{restatable}{theorem}{restateFourTravGen}\label{p:4-trav-gen}
		Let $P$ be a pattern that contains four 1-entries $x_1, x_2, x_3, x_4$ such that for \linebreak each~${i \in [4]}$, there are no other 1-entries in the same row or column as $x_i$, and $x_i$ is in the first or last row or column,
		and $x_1, x_2, x_3, x_4$ form one of the two patterns
		\begin{align*}
			\begin{smallbulletmatrix}
				&\o&  &  \\
				&  &  &\o\\
				\o&  &  &  \\
				&  &\o&  
			\end{smallbulletmatrix},
			\begin{smallbulletmatrix}
				&  &\o&  \\
				\o&  &  &  \\
				&  &  &\o\\
				&\o&  &  
			\end{smallbulletmatrix}.
		\end{align*}
		Then $\sat(P,n) \in \fO(1)$.
	\end{restatable}
	
	In \cref{sec:witnesses}, we define and discuss (vertical) witnesses, and in \cref{sec:spanosc}, we define spanning oscillations. In \cref{sec:structure}, we present the structure of the proof of \cref{p:main}. In \cref{sec:embeddings} we introduce an alternative characterization of pattern containment that simplifies our proofs.
	
	\Cref{sec:construct_examples} gives an introduction to the witness-construction techniques used in the following chapters. In \cref{sec:4-trav,sec:surround,sec:complicated}, we construct vertical witnesses for all permutation matrices, based on different types of spanning oscillations, which proves \cref{p:main}. We also prove \cref{p:4-trav-gen} in \cref{sec:4-trav}.
	
	We now introduce conventions and notations used throughout the paper. Some more definitions needed for \cref{sec:4-trav,sec:surround,sec:complicated} will be introduced in \cref{sec:embeddings}.
	
	We identify 1-entries in an $m \times n$ matrix $M$ as their positions $(i,j) \in [m] \times [n]$, where $i$ is the row of the 1-entry (from top to bottom), and $j$ is its column (from left to right). $E(M)$ denotes the set of 1-entries in $M$. For two 1-entries $x = (i,j) \in E(M)$ and $x' = (i', j') \in E(M)$, we write $x \ltv x'$ if $i<i'$ and $x \lth x'$ if $j<j'$. Define $x \lev x'$ and $x \leh x'$ analogously. We also say $x$ is \emph{above} $x'$ if $x \ltv x'$, and use \emph{below}, \emph{to the right}, and \emph{to the left} similarly.
	
	In a permutation matrix $P$, we denote the leftmost (rightmost, topmost, bottommost) 1-entry of $P$ by $\ell_P$ ($r_P$, $t_P$, $b_P$). Note that if $P$ is an indecomposable $k \times k$ permutation matrix with~$k > 1$, then these four 1-entries are pairwise distinct.
	
	Let $M$ be an arbitrary matrix. Denote by $\rot(M)$ the matrix obtained by rotating $M$ 90 degrees clockwise, denote by $\rev(M)$ the matrix obtained by reversing all rows of $M$, and denote by $\trans(M)$ the transpose of $M$, i.e., the matrix obtained by swapping the roles of rows and columns.\footnote{We do not use the common superscript $^\rT$, as it will later be used with the meaning ``top''.}
	
	If $A$ is a $k \times m$ matrix and $B$ is a $k \times n$ matrix, then the \emph{horizontal concatenation}\label{def:concatenation} $(A,B)$ is the $k \times (m+n)$ matrix $M$ where $E(M) = E(A) \cup \{(i,j+m) \mid (i,j) \in E(B)\}$. Intuitively, $M$ is obtained by placing $A$ to the left of $B$. The horizontal concatenation $(A_1, A_2, \dots)$ of a sequence of matrices with the same height is defined accordingly.

	\subsection{Witnesses}\label{sec:witnesses}
	
	Let $P$ be a matrix pattern without empty rows or columns. An \emph{explicit witness}\footnote{An explicit witness is what Fulek and Keszegh~\cite{FulekKeszegh2021} simply call a \emph{witness}.} for $P$ is a matrix~$M$ that is $P$-saturating and contains at least one empty row and at least one empty column.
	
	\begin{lemma}[{\cite{FulekKeszegh2021}}]\label{p:const-then-witness}
		For each pattern $P$ without empty rows and columns, we \linebreak have ${\sat(P,n) \in \fO(1)}$ if and only if $P$ has an explicit witness.
	\end{lemma}
	\begin{proof}
		Suppose $\sat(P,n) \le c_P$ for all $n \in \N$. Then there exists a $(c_P+1) \times (c_P+1)$ $P$-saturating matrix $M$ with at most $c_P$ 1-entries. Clearly, $M$ has an empty row and an empty column, so $M$ is an explicit witness for $P$.
		
		Now suppose that $P$ has an $m_0 \times n_0$ explicit witness $M$ of weight $w$. We can replace the empty row (column) in $M$ by an arbitrary number of empty rows (columns), and the resulting (arbitrarily large) matrix will still be $P$-saturating. As such, $\sat(P,m,n) \le w$ for all $m \ge m_0$ and $n \ge n_0$. Note that it is critical here that $P$ has no empty rows or columns. Otherwise, inserting empty rows or columns into $M$ might create an occurrence of $P$.
	\end{proof}
	
	We call a row (column) of a matrix $M$ \emph{$P$-expandable} if the row (column) is empty and adding a single 1-entry anywhere in that row (column) creates a new occurrence of $P$ in $M$. An explicit witness for $P$ is thus a saturating matrix with at least one $P$-expandable row and at least one $P$-expandable column. We define a \emph{witness} for $P$ (used implicitly by Fulek and Keszegh) as a matrix that avoids $P$ and has at least one $P$-expandable row and at least one $P$-expandable column. Clearly, an explicit witness is a witness. The following lemma shows that finding a witness is sufficient to show that $\sat(P,n) \in \fO(1)$.
	
	\begin{lemma}\label{p:gwitness}
		If a pattern $P$ without empty rows or columns has an $m_0 \times n_0$ witness, then $P$ has an $m_0 \times n_0$ explicit witness.
	\end{lemma}
	\begin{proof}
		Let $M$ be an $m_0 \times n_0$ witness for $P$. If $M$ is $P$-saturating, then we are done. Otherwise, there must be a 0-entry $(i,j)$ in $M$ that can be changed to 1 without creating an occurrence $P$. Choose one such 0-entry and turn it into 1. Note that $(i,j)$ cannot be contained in an expandable row or column of $M$, so the resulting matrix is still a witness. Thus, we obtain an explicit witness after repeating this step at most $m_0 \cdot n_0$ times.
	\end{proof}
	
	\subsubsection{Vertical and horizontal witnesses}
	
	Fulek and Keszegh also considered the asymptotic behavior of the functions $\sat(P,m_0,n)$ \linebreak and $\sat(P,m,n_0)$, where $m_0$ and $n_0$ are fixed. The dichotomy of $\sat(P,n)$ also holds in this setting:
	
	\begin{theorem}[{\cite[Parts of Theorem 1.3]{FulekKeszegh2021}}]\label{p:dichotomy}
		For every pattern $P$, and constants $m_0, n_0$,
		\begin{enumerate}[(i)]
			\item either $\sat(P,m_0,n) \in \fO(1)$ or $\sat(P,m_0,n) \in \Theta(n)$;\label{item:dich_hor}
			\item either $\sat(P,m,n_0) \in \fO(1)$ or $\sat(P,m,n_0) \in \Theta(m)$.\label{item:dich_vert}
		\end{enumerate}
	\end{theorem}
	
	We can adapt the notion of witnesses in order to classify $\sat(P,m_0, n)$ and $\sat(P,m,n_0)$. Let~$P$ be a matrix pattern without empty rows or columns. A \emph{horizontal (vertical) witness} for~$P$ is a matrix $M$ that avoids $P$ and contains an expandable column (row).\footnote{A horizontal witness can be expanded horizontally, a vertical witness can be expanded vertically.} Clearly, $P$ has a horizontal witness with $m_0$ rows if and only if $\sat(P, m_0, n)$ is bounded; and $P$ has a vertical witness with~$n_0$ columns if and only if $\sat(P,m,n_0)$ is bounded. Further note that $M$ is a witness for $P$ if and only if $M$ is both a horizontal witness and a vertical witness.
	
	We now prove that we can essentially restrict our attention to the classification of the functions $\sat(P,m_0, n)$ and $\sat(P,m,n_0)$. The following two lemmas are a generalization of the technique used by Fulek and Keszegh to prove that $\sat(Q,n) \in \fO(1)$ for the pattern $Q$ depicted in \cref{fig:mat-FK}.
	
	\begin{lemma}\label{p:ext-hor-wit}
		Let $P$ be a matrix pattern without empty rows or columns, and only one 1-entry in the last row (column). Let $W$ be a horizontal (vertical) witness for $P$. Then, appending an empty row (column) to $W$ again yields a horizontal (vertical) witness.
	\end{lemma}
	\begin{proof}
		We prove the lemma for horizontal witnesses with a row appended. The other case follows by symmetry.
		Let $W$ be an $m_0 \times n_0$ horizontal witness for $P$, where the $j$-th column of~$W$ is expandable. Let $W'$ be the matrix obtained by appending a row to $W$. Clearly, $W'$ still does not contain $P$. Moreover, adding an entry in $W'$ at $(i,j)$ for any $i \neq n_0+1$ creates a new occurrence of $P$. It remains to show that adding an entry at $(n_0+1,j)$ creates an occurrence of~$P$.
		
		We know that adding an entry at $(n_0,j)$ in $W'$ creates an occurrence of $P$, say at \linebreak positions~$I \subseteq [n_0]^2$. Since $P$ has only one entry in the last row, all positions $(i', j') \in I \setminus \{(n_0,j)\}$ satisfy $i' < n_0+1$. Thus, adding a 1-entry at $(n_0+1,j)$ instead of $(n_0,j)$ creates an occurrence of $P$ at positions $I \setminus \{(n_0,j)\} \cup \{(n_0+1,j)\}$. Thus, $W'$ is a horizontal witness.
	\end{proof}
	
	\begin{lemma}\label{p:vert_hor_wit}
		Let $P$ be a indecomposable pattern without empty rows or columns, and with only one 1-entry in the last row and one 1-entry in the last column. Then $\sat(P,n) \in \fO(1)$ if and only if there exist constants $m_0, n_0$ such that $\sat(P,m_0, n) \in \fO(1)$ and $\sat(P,m, n_0) \in \fO(1)$.
	\end{lemma}
	\begin{proof}
		Suppose that $\sat(P,n) \in \fO(1)$. Then $P$ has an $m_0 \times n_0$ witness $M$ by \cref{p:const-then-witness}, and thus $\sat(P,m_0, n)$ is at most the weight of $M$, for every $n \ge n_0$. Similarly, ${\sat(P,m, n_0) \! \in \! \fO(1)}$.
		
		Now suppose that $\sat(P,m_0, n) \in \fO(1)$ and $\sat(P,m, n_0) \in \fO(1)$. Then, for some $m_1, n_1$, there exists an $m_0 \times n_1$ horizontal witness $W_\rH$ and an $m_1 \times n_0$ vertical witness $W_\rV$. Consider the following $(m_0+m_1) \times (n_0+n_1)$ matrix, where $\zeromat_{m \times n}$ denotes the all-0 $m \times n$ matrix:
		\begin{align*}
			W = \begin{pmatrix} \zeromat_{m_0 \times n_0} & W_\rH \\ W_\rV & \zeromat_{m_1 \times n_1} \end{pmatrix}
		\end{align*}
		
		We first show that $W$ does not contain $P$. Suppose it does. Since $P$ is contained neither in~$W_\rH$ nor in $W_\rV$, an occurrence of $P$ in $W$ must contain 1-entries in both the bottom left and top right quadrant. But then $P$ is decomposable, a contradiction.
		
		By \Cref{p:ext-hor-wit}, $W_\rV' = (W_\rV, \zeromat_{m_1 \times n_1})$ is a vertical witness, and $W_\rH' = \binom{W_\rH}{\zeromat_{m_1 \times n_1}}$ is a horizontal witness. The expandable row in $W_\rV'$ and the expandable column in $W_\rH'$ are both also present in~$W$. This implies that $W$ is a witness for $P$, so $\sat(P, n) \in \fO(1)$.
	\end{proof}
	
	\Cref{fig:wit-genwit} shows an example of a witness, constructed with \Cref{p:vert_hor_wit}, using vertical and horizontal witnesses presented later in \cref{sec:4-trav}, and an explicit witness constructed using \cref{p:gwitness}.
	
	\begin{figure}
		\centering
		\begin{align*}
			\begin{smallbulletmatrix}
				  &  &\o&  \\
				\o&  &  &  \\
				  &  &  &\o\\
				  &\o&  &  
			\end{smallbulletmatrix}
			\hspace{10mm}
			\begin{smallbulletmatrix}
				  &  &  &  &  &  &  &  &\d&\o&  \\
				  &  &  &  &  &  &  &\o&\d&  &  \\
				  &  &  &  &  &  &  &  &\d&  &\o\\
				  &  &  &  &  &  &\o&  &\d&  &  \\
				  &  &  &  &  &  &  &  &\d&\o&  \\
				  &  &  &  &  &  &  &\o&\d&  &  \\
				  &  &\o&  &  &  &  &  &\d&  &  \\
				\o&  &  &  &\o&  &  &  &\d&  &  \\
				\d&\d&\d&\d&\d&\d&\d&\d&\d&\d&\d\\
				  &\o&  &  &  &\o&  &  &\d&  &  \\
				  &  &  &\o&  &  &  &  &\d&  &  
			\end{smallbulletmatrix}
			\hspace{10mm}
			\begin{smallbulletmatrix}
				\o&\o&\o&\o&  &\o&\o&\o&\d&\o&\o\\
				  &  &  &  &  &  &  &\o&\d&\o&\o\\
				  &  &  &  &  &\o&\o&\o&\d&\o&\o\\
				  &  &  &  &  &\o&\o&\o&\d&\o&  \\
				  &  &  &  &  &\o&  &\o&\d&\o&  \\
				  &\o&\o&\o&  &\o&  &\o&\d&\o&\o\\
				  &\o&\o&\o&  &\o&  &  &\d&  &  \\
				\o&\o&  &\o&\o&\o&  &  &\d&  &  \\
				\d&\d&\d&\d&\d&\d&\d&\d&\d&\d&\d\\
				\o&\o&  &\o&\o&\o&  &  &\d&  &  \\
				\o&  &  &\o&\o&\o&  &  &\d&\o&\o
			\end{smallbulletmatrix}
		\end{align*}
		\caption{A pattern (left), a witness (middle) and an explicit witness (right) for the pattern. The small dots indicate the expandable row/column.}\label{fig:wit-genwit}
	\end{figure}
	
	Observe that the transformations $\rev$, $\rot$, and $\trans$ all preserve witnesses. However, the latter two change vertical witnesses to horizontal witnesses, and vice versa. Formally:
	
	\begin{observation}\label{p:transform_effects}
		Let $P$ be a matrix with a vertical witness $W$. Then $\rev(W)$ is a vertical witness of $\rev(P)$, $\rot(W)$ is a horizontal witness of $\rot(P)$, and $\trans(W)$ is a horizontal witness of $\trans(P)$.\qed
	\end{observation}
	
	Recall that our goal is to show that every indecomposable permutation matrix has a witness. Since indecomposable permutation matrices are closed under transposition, \cref{p:vert_hor_wit,p:transform_effects} imply that it suffices to find a \emph{vertical} witness for each indecomposable permutation matrix. The same is true for every class of matrices satisfying the conditions of \cref{p:vert_hor_wit} that is closed under transposition or 90-degree clockwise rotation. This is useful to prove \cref{p:4-trav-gen}.
	
	\begin{lemma}\label{p:vert_wit_suff}
		Let $\fP$ be a class of indecomposable patterns without empty rows or columns, and with only one 1-entry in the last row and one 1-entry in the last column. If $\fP$ is closed under transposition or 90-degree clockwise rotation and each pattern in $\fP$ has a vertical witness, then~$\sat(P,n) \in \fO(1)$ for each $P \in \fP$.
	\end{lemma}
	\begin{proof}
		Suppose that $\fP$ is closed under transposition and each $P \in \fP$ has a vertical witness. By \cref{p:vert_hor_wit}, it suffices to show that each pattern in $\fP$ also has a horizontal witness. Let $P \in \fP$. Then $\trans(P) \in \fP$ has a vertical witness $W$. By \cref{p:transform_effects}, $\trans(W)$ is a horizontal witness for $\trans(\trans(P)) = P$.
		
		The case that $\fP$ is closed under 90-degree rotation can be handled analogously.
	\end{proof}
	
	\subsection{Spanning oscillations}\label{sec:spanosc}
	
	We now introduce \emph{spanning oscillations}, a class of substructures that characterizes indecomposable permutation matrices.
	
	For a permutation matrix $P$, the \emph{permutation graph} $G_P$ of the underlying permutation can be defined as follows: The vertex set is $E(P)$, and two 1-entries $x, y \in E(P)$ have an edge between them if $x$ is below and to the left of $y$ (or vice versa).
	
	An \emph{oscillation} in a permutation matrix of $P$ is a sequence $X = (x_1, x_2, \dots, x_m)$ of distinct 1-entries in $P$ such that $X$ forms an induced path in $G_P$, i.e., there is an edge between $x_i$ and~$x_{i+1}$ for each $i \in [m-1]$, and no other edges between 1-entries in $X$. Oscillations have been studied before in several contexts \cite{Pratt1973,BrignallEtAl2008,Vatter2011}.
	Vatter showed that a permutation matrix $P$ is sum indecomposable if and only if it has an oscillation that starts with $\ell_P$ and ends with $r_P$~\cite[Propositions 1.4, 1.7]{Vatter2011}. Our characterization of indecomposable permutations is very similar. Call an oscillation $X = (x_1, x_2, \dots, x_m)$ \emph{spanning} if $\{x_1, x_2\} = \{\ell_P, t_P\}$ and~$\{x_{m-1}, x_m\} = \{b_P, r_P\}$.
	
	\begin{lemma}\label{p:indec-spanosc}
		Let $P$ be a sum indecomposable permutation matrix such that $t_P$ is to the left of~$b_P$ or $\ell_P$ is above $r_P$. Then $P$ has a spanning oscillation.
	\end{lemma}
	\begin{proof}
		We write $\ell, t, b, r$ for $\ell_P, t_P, b_P, r_P$. By symmetry, we can assume that $t$ is to the left of $b$ (otherwise, replace $P$ by $\trans(P)$, noting that $G_P = G_{\trans(P)}$). Recall that  $\ell, t, b, r$ are pairwise distinct, as $P$ is indecomposable and not $1 \times 1$.
		
		Since $P$ is sum indecomposable, by the result of Vatter mentioned above, it has an oscilla\-tion~$X' = (x_1', x_2', \dots, x_m')$ with $x_1' = \ell$, $x_m' = r$. Suppose first that $t$ occurs in $X'$. Since~$G_P$ has an edge between $\ell$ and $t$, and $X$ is an \emph{induced} path in $G_P$, this means that $x_2' = t$. Otherwise, note that $t$ is connected in $G_P$ to precisely those 1-entries that are to the left of $t$. Let $i$ be maximal such that $x_i'$ is to the left of $t$. If $i = 1$, then $(t, \ell, x_2', \dots, x_m')$ is an induced path in~$G_P$. Otherwise, $\ell, t, x_i', \dots, x_m'$ is an induced path in~$G_P$. In either case, we have an oscilla\-tion~$X'' = (x_1'', x_2'', \dots, x_m'')$ that starts with $\{\ell, t\}$ and ends with $r$.
		
		It remains to make sure that $b$ is among the last two 1-entries in the oscillation. If $b$ occurs in~$X''$, then $x_{m-1}'' = b$, since $X''$ is an induced path. Otherwise, let $j$ be minimal such that~$x_j$ is to the right of $b$. If $j = m$, then $X = (x_1'', x_2'', \dots, x_{m-1}'', r, b)$ is an induced path in $G_P$. Otherwise, $X = (x_1'', x_2'', \dots, x_j'', b, r)$ is an induced path in $G_P$. Since $\ell, t$ are both to the left of~$b$, we have~$j > 2$, so $X$ is a spanning oscillation.
	\end{proof}
	
	We obtain the following characterization of indecomposable permutation matrices.
	\begin{corollary}\label{p:indec-equiv}
		A permutation matrix $P$ is indecomposable if and only if $P$ or $\rev(P)$ has a spanning oscillation or $P$ is the $1 \times 1$ permutation matrix.
	\end{corollary}
	\begin{proof}
		First, assume $P$ is indecomposable. If $t_P$ is to the left of $b_P$, then \cref{p:indec-spanosc} implies that $P$ has a spanning oscillation. If $t_P$ is to the right of $b_P$, then \cref{p:indec-spanosc} implies that $\rev(P)$ has a spanning oscillation. If $t_P = b_P$, then $P$ is $1 \times 1$.
		
		Second, assume $P$ has a spanning oscillation. Then $P$ is sum indecomposable. Suppose~$P$ is decomposable, then $P$ has the form $\rdecmat$, so $t$ is to the right of $b$ and $\ell$ is below $r$. But then $\ell, b, t, r$ form the complete bipartite graph $K_{2,2}$ in $G_P$, implying that $P$ has no spanning oscillation, a contradiction. A symmetric argument shows that $P$ is indecomposable if $\rev(P)$ has a spanning oscillation.
	\end{proof}
	
	Spanning oscillations have a very rigid structure, which we now describe more explicitly, in terms of relative positions of 1-entries. Let $P$ be a permutation matrix and $X = (x_1, x_2, \dots, x_m)$ be a spanning oscillation of $P$. For $2 \le i \le m-1$, call $x_i$ an \emph{upper} 1-entry if $x_i$ is above and to the right of $x_{i-1}$ and $x_{i+1}$, and call $x_i$ a \emph{lower} 1-entry if $x_i$ is below and to the left of $x_{i-1}$ and~$x_{i+1}$. Since $G_P$ contains the edges $\{x_{i-1}, x_i\}$ and $\{x_i, x_{i+1}\}$, but not the edge $\{x_{i-1}, x_{i+1}\}$, every 1-entry (except $x_1, x_m$) is either upper or lower. Clearly, upper and lower 1-entries alternate, i.e.,~$x_i$~is upper if and only if $x_{i+1}$ is lower, for $2 \le i < m-1$. It is convenient to also call $\ell_P, b_P$ lower 1-entries and $t_P, r_P$ upper 1-entries. We then have:
	
	\begin{observation}\label{p:spanosc-upper-lower}
		Let $P$ be a permutation matrix and $X = (x_1, x_2, \dots, x_m)$ be a spanning oscillation of $P$. If $x_1 = \ell_P$, then all $x_i$ with odd $i$ are lower 1-entries, and all $x_i$ with even $i$ are upper 1-entries. If $x_1 = t_P$, then all $x_i$ with odd $i$ are upper 1-entries, and all $x_i$ with even $i$ are lower 1-entries.\qed
	\end{observation}
	
	It is easy to see that, if $x_1 = \ell_P$, then $x_3$, $x_4$ must be below and to the right of $x_1$. By induction, and by considering symmetric cases, we can prove:
	\begin{observation}\label{p:spanosc-rigidity}
		Let $P$ be a permutation matrix and $X = (x_1, x_2, \dots, x_m)$ be a spanning oscillation of $P$. Then $x_i$ is above and to the left of $x_j$ for each $i \in [m-2]$ and $i+2 \le j \le m$.
	\end{observation}
	
	This leaves us with only two possible forms of spanning oscillations for each length $m$, see \cref{fig:min-spanosc}. Observe that spanning oscillations are preserved by transposition and 180-degree rotation, in the following sense.
	
	Let $P$ be a $k \times k$ permutation matrix and $X = (i_1,j_1), (i_2,j_2), \dots, (i_\ell,j_\ell)$ a spanning oscillation of $P$. Define
	\begin{align*}
		& \trans(X) = (j_1, i_1), (j_2, i_2), \dots, (j_\ell,i_\ell), \text{ and}\\
		& \rot^2(X) = (k-i_\ell,k-j_\ell), (k-i_{\ell-1},k-j_{\ell-1}), \dots, (k-i_1, k-j_1).
	\end{align*}

	It is easy to see that $\trans(X)$ is a spanning oscillation of $\trans(P)$ and $\rot^2(X)$ is a spanning oscillation of $\rot^2(P)$.
	
	\begin{figure}
		\centering
		\begin{tikzpicture}[
				scale=0.3,
				graph/.style={densely dashed}
			]
			\footnotesize
			\begin{scope}
				\draw (1,-1) rectangle (4,-4);
				\tNamedPoint{1,-2}{left}{$x_1$}
				\tNamedPoint{3,-1}{above}{$x_2$}
				\tNamedPoint{2,-4}{below}{$x_3$}
				\tNamedPoint{4,-3}{right}{$x_4$}
				\draw[graph] (1,-2) -- (3,-1) -- (2,-4) -- (4,-3);
			\end{scope}
			\begin{scope}[shift={(7,0.5)}]
				\draw (1,-1) -- (5,-1) -- (5,-5) -- (3,-5) -- (3,-4) -- (1,-4) -- (1,-1);
				\tNamedPoint{1,-2}{left}{$x_1$}
				\tNamedPoint{3,-1}{above}{$x_2$}
				\tNamedPoint{2,-4}{below}{$x_3$}
				\tNamedPoint{5,-3}{right}{$x_4$}
				\tNamedPoint{4,-5}{below}{$x_5$}
				\draw[graph] (1,-2) -- (3,-1) -- (2,-4) -- (5,-3) -- (4,-5);
			\end{scope}
			\begin{scope}[shift={(15,1)}]
				\draw (1,-1) -- (4,-1) -- (4,-3) -- (6,-3) -- (6,-6) -- (3,-6) -- (3,-4) -- (1,-4) -- (1,-1);
				\tNamedPoint{1,-2}{left}{$x_1$}
				\tNamedPoint{3,-1}{above}{$x_2$}
				\tNamedPoint{2,-4}{below}{$x_3$}
				\tNamedPoint{5,-3}{above}{$x_4$}
				\tNamedPoint{4,-6}{below}{$x_5$}
				\tNamedPoint{6,-5}{right}{$x_6$}
				\draw[graph] (1,-2) -- (3,-1) -- (2,-4) -- (5,-3) -- (4,-6) -- (6,-5);
			\end{scope}
			\begin{scope}[shift={(24,1.5)}]
				\draw (1,-1) -- (4,-1) -- (4,-3) -- (7,-3) -- (7,-7) -- (5,-7) -- (5,-6) -- (3,-6) -- (3,-4) -- (1,-4) -- (1,-1);
				\tNamedPoint{1,-2}{left}{$x_1$}
				\tNamedPoint{3,-1}{above}{$x_2$}
				\tNamedPoint{2,-4}{below}{$x_3$}
				\tNamedPoint{5,-3}{above}{$x_4$}
				\tNamedPoint{4,-6}{below}{$x_5$}
				\tNamedPoint{7,-5}{right}{$x_6$}
				\tNamedPoint{6,-7}{below}{$x_7$}
				\draw[graph] (1,-2) -- (3,-1) -- (2,-4) -- (5,-3) -- (4,-6) -- (7,-5) -- (6,-7);
			\end{scope}
			\begin{scope}[shift={(0,-8)}]
				\begin{scope}
					\draw (1,-1) rectangle (4,-4);
					\tNamedPoint{2,-1}{above}{$x_1$}
					\tNamedPoint{1,-3}{left}{$x_2$}
					\tNamedPoint{4,-2}{right}{$x_3$}
					\tNamedPoint{3,-4}{below}{$x_4$}
					\draw[graph] (2,-1) -- (1,-3) -- (4,-2) -- (3,-4);
				\end{scope}
				\begin{scope}[shift={(7,0.5)}]
					\begin{scope}[yscale=-1, shift={(0,6)}, rotate around={90:(3,-3)}]
						\draw (1,-1) -- (5,-1) -- (5,-5) -- (3,-5) -- (3,-4) -- (1,-4) -- (1,-1);
					\end{scope}
					\tNamedPoint{2,-1}{above}{$x_1$}
					\tNamedPoint{1,-3}{left}{$x_2$}
					\tNamedPoint{4,-2}{right}{$x_3$}
					\tNamedPoint{3,-5}{below}{$x_4$}
					\tNamedPoint{5,-4}{right}{$x_5$}
					\draw[graph] (2,-1) -- (1,-3) -- (4,-2) -- (3,-5) -- (5,-4);
				\end{scope}
				\begin{scope}[shift={(15,1)}]
					\begin{scope}[yscale=-1, shift={(0,7)}, rotate around={90:(3.5,-3.5)}]
						\draw (1,-1) -- (4,-1) -- (4,-3) -- (6,-3) -- (6,-6) -- (3,-6) -- (3,-4) -- (1,-4) -- (1,-1);
					\end{scope}
					\tNamedPoint{2,-1}{above}{$x_1$}
					\tNamedPoint{1,-3}{left}{$x_2$}
					\tNamedPoint{4,-2}{right}{$x_3$}
					\tNamedPoint{3,-5}{left}{$x_4$}
					\tNamedPoint{6,-4}{right}{$x_5$}
					\tNamedPoint{5,-6}{below}{$x_6$}
					\draw[graph] (2,-1) -- (1,-3) -- (4,-2) -- (3,-5) -- (6,-4) -- (5,-6);
				\end{scope}
				\begin{scope}[shift={(24,1.5)}]
					\begin{scope}[yscale=-1, shift={(0,8)}, rotate around={90:(4,-4)}]
						\draw (1,-1) -- (4,-1) -- (4,-3) -- (7,-3) -- (7,-7) -- (5,-7) -- (5,-6) -- (3,-6) -- (3,-4) -- (1,-4) -- (1,-1);
					\end{scope}
					\tNamedPoint{2,-1}{above}{$x_1$}
					\tNamedPoint{1,-3}{left}{$x_2$}
					\tNamedPoint{4,-2}{right}{$x_3$}
					\tNamedPoint{3,-5}{left}{$x_4$}
					\tNamedPoint{6,-4}{right}{$x_5$}
					\tNamedPoint{5,-7}{below}{$x_6$}
					\tNamedPoint{7,-6}{right}{$x_7$}
					\draw[graph] (2,-1) -- (1,-3) -- (4,-2) -- (3,-5) -- (6,-4) -- (5,-7) -- (7,-6);
				\end{scope}
			\end{scope}
		\end{tikzpicture}
		\caption{The spanning oscillations of length $m$, for $m = 4, 5, 6, 7$. The dashed line segments indicate the edges of the permutation graph. The borders indicate the possible positions for other 1-entries if the spanning oscillation is tall (top row) or wide (bottom row).}\label{fig:min-spanosc}
	\end{figure}
	
	A spanning oscillation $X = (x_1, x_2, \dots, x_m)$ is \emph{tall} if the following two properties are satisfied for each $2 \le i \le m-2$ where $x_i$ is an upper 1-entry.
	\begin{enumerate}[(i)]
		\item $P$ has no 1-entry that is below $x_{i+1}$ and to the left of $x_i$.\label{prop:tall-below}
		\item $P$ has no 1-entry that is above $x_i$ and to the right of $x_{i+1}$.\label{prop:tall-above}
		\setcounter{resumeEnum}{\value{enumi}}
	\end{enumerate}

	A spanning oscillation $X$ is \emph{wide} if $\trans(X)$ is tall. We now show that we can always assume that a minimum-length spanning oscillation is tall (or wide).
	
	\begin{lemma}\label{p:min-osc-tall}
		Let $P$ be a permutation matrix and $X = (x_1, x_2, \dots, x_m)$ be a spanning oscillation of $P$ of minimum length $m$. Then $P$ has a tall spanning oscillation of length $m$ that starts with $x_1, x_2$ and ends with $x_{m-1}, x_m$.
	\end{lemma}
	\begin{proof}
		Suppose $X$ is not tall, so it violates (\ref{prop:tall-below}) or (\ref{prop:tall-above}) at some index $i$ with $2 \le i \le m-2$. We now show how to construct a spanning oscillation $X'$ of length $m$ that starts with $x_1, x_2$, ends with $x_{m-1}, x_m$, and violates (\ref{prop:tall-below}) or (\ref{prop:tall-above}) less often than $X$. Repeating this, we eventually obtain a tall spanning oscillation.
		
		Suppose first that $X$ violates (\ref{prop:tall-below}) at index $i$. Then $x_i$ is an upper 1-entry, and there is a~${y \! \in \! E(P)}$ such that $y$ is below $x_{i+1}$ and to the left of $x_i$. Assume $y$ is the bottommost such 1-entry. Note~${y \notin \{\ell_P, b_P\}}$, and that $x_{i+2}$ is above $x_{i+1}$ by \cref{p:spanosc-upper-lower}.
		
		Let $j$ be minimal such that $x_j$ is to the right of $y$. Since $\ell_P \lth y \lth x_i$, we have $2 \le j \le i$. Let $k$ be maximal such that $x_k$ is above $y$. Since $x_{i+2} \ltv y \ltv b_P$, we have $i+2 \le k \le m-1$.
		
		By \cref{p:spanosc-rigidity}, $x_j$ is above and to the left of $x_k$, meaning that both $x_j$ and $x_k$ are above and to the right of $y$. Thus, the sequence $X' = (x_1, x_2, \dots, x_j, y, x_k, x_{k+1}, \dots, x_m)$ is a path. We now show that $X'$ is an \emph{induced} path. Let $j' < j$. By definition of $j$, we know that $x_{j'}$ is to the left of $y$. By \cref{p:spanosc-rigidity}, $x_{j'}$ is above $x_{i+1}$, implying that $x_{j'}$ is above $y$. Thus, $G_P$ has no edge between $x_{j'}$ and $y$. Similarly, we can prove that there is no edge between $y$ and $x_{k'}$ for each $k' > k$.
		
		Since $2 \le j$ and $k \le m-1$, we know that $X'$ starts with $x_1, x_2$ and ends with $x_{m-1}, x_m$, implying that $X'$ is a spanning oscillation.
		
		By assumption, $P$ has no spanning oscillation shorter than $P$, so $X$ must have length $m$, implying that $j = i$ and $k = i+2$.  Further, $X'$ does not violate (\ref{prop:tall-below}) at index $i$, since, by choice of $y$, there are no 1-entries below $y$ and to the left of $x_j = x_i$. Thus, $X'$ has strictly less overall violations of (\ref{prop:tall-below}) or (\ref{prop:tall-above}) than $X$.
		
		The second case, where $X$ violates (\ref{prop:tall-above}), can be proven symmetrically.
	\end{proof}
	
	Clearly, the statement of \cref{p:min-osc-tall} is also true when replacing ``tall'' with ``wide'', using the same proof on $\trans(P)$.
	
	\subsection{Structure of the main proof}\label{sec:structure}
	
	We divide the proof of \cref{p:main} into three cases, proven in \cref{sec:4-trav,sec:surround,sec:complicated}. In \cref{sec:4-trav}, we handle the special case of length-4 spanning oscillations:
	\begin{lemma}\label{p:4-osc}
		Each permutation matrix with a spanning oscillation of length 4 has a vertical witness.
	\end{lemma}
	
	In \cref{sec:surround}, we prove:
	\begin{restatable}{lemma}{restateInvTravWit}\label{p:5-t-osc}
		Each permutation matrix $P$ with a wide spanning oscillation of length $m \ge 5$ that starts with $t_P$ has a vertical witness.
	\end{restatable}
	
	The final and most involved case is treated in \cref{sec:complicated}:
	\begin{restatable}{lemma}{restateTravWit}\label{p:even-6-l-osc}
		Each permutation matrix $P$ with a tall spanning oscillation of even length $m \ge 6$ that starts with $\ell_P$ has a vertical witness.
	\end{restatable}
	
	It is not immediately obvious that \cref{p:4-osc,p:5-t-osc,p:even-6-l-osc} cover all indecomposable permutation matrices. We now show that this is the case.
	\begin{corollary}\label{p:main-wit}
		Every indecomposable permutation matrix has a vertical witness.
	\end{corollary}
	\begin{proof}
		Let $P$ be an indecomposable permutation matrix. If $P$ is $1 \times 1$, any all-zero matrix is a witness of $P$. Otherwise, one of $P$ and $\rev(P)$ has a spanning oscillation $X$ by \cref{p:indec-equiv}. By \cref{p:transform_effects}, it suffices to find a vertical witness for either $P$ or $\rev(P)$, so without loss of generality, assume that $X$ is a spanning oscillation of $P$, and that $X$ has minimum length~$m$. If~$m = 4$, we can apply \cref{p:4-osc}. If $m \ge 5$ and $X$ starts with $t_P$, then \cref{p:min-osc-tall} implies that $P$ also has a wide spanning oscillation of size $m$ that starts with $t_P$, so we can apply \cref{p:5-t-osc}.
		
		Now assume $m \ge 5$ and $X$ starts with $\ell_P$. If $m$ is even, we can apply \cref{p:even-6-l-osc}, since by \cref{p:min-osc-tall} we can assume that $X$ is tall. Otherwise, if $m$ is odd, \cref{p:spanosc-upper-lower} implies that~$X$ ends with $b_P$. This means that the spanning oscillation $\rot^2(X)$ of $\rot^2(P)$ starts with~$t_{\rot^2(P)}$, so we can apply \cref{p:5-t-osc} to obtain a witness $W'$ of $\rot^2(P)$. \Cref{p:transform_effects} implies that $\rot^2(W')$ is a witness of $P$.
	\end{proof}

	\subsection{Embeddings}\label{sec:embeddings}
	
	In the following sections, we use an alternative definition of pattern containment based on sets of 1-entries. Let $P$ be a pattern and $M$ be a matrix. We say a function ${\phi \colon E(P) \rightarrow E(M)}$ \linebreak is an \emph{embedding} of $P$ into $M$ if for $x, y \in E(P)$ we have $x \lth y \Leftrightarrow \phi(x) \lth \phi(y)$ \linebreak and~$x \ltv y \Leftrightarrow \phi(x) \ltv \phi(y)$.
	
	Note that if we allow empty rows or columns in $P$, then $E(P)$ does not determine $P$, since appending an empty row or column to $P$ does not change $E(P)$. This means that the existence of an embedding of $P$ into $M$ does not necessarily imply that $P$ is contained in $M$. However, we only consider patterns without empty rows or columns in this paper, and in that case, equivalence holds.
	\begin{restatable}{lemma}{restateEquivContainment}\label{p:equiv-containment}
		Let $P$, $M$ be matrices, and let $P$ have no empty rows or columns. Then $P$ is contained in $M$ if and only if there is an embedding of $P$ into $M$.
	\end{restatable}
	
	A proof of \cref{p:equiv-containment} is provided in \cref{app:proof-equiv-containment}. We now introduce some notations used in the following sections.
	
	Let $x =(i,j)$, $y = (i',j')$ be two 1-entries. The \emph{horizontal distance} between $x$ and $y$ is~$\hd( x, y) = |i-i'|$, and the \emph{vertical distance} between $x$ and $y$ is $\vd( x, y) = |j-j'|$. The \emph{width} $\width(A)$ (resp. \emph{height} $\height(A)$) of a set $A \subseteq E(M)$ is the maximum horizontal (resp. vertical) distance between two 1-entries in $A$.
	
	Let $\phi$ be an embedding of $P$ into $M$. We say a row (column) is \emph{hit} by $\phi$ if $\phi(x)$ is in that row (column) for some 1-entry $x \in E(P)$.
	We define variants of the above notions that only ``count'' rows and columns of $M$ that are hit by $\phi$. This will be useful when we have partial information about $\phi$, or when we know that certain rows/columns are empty and thus cannot be hit by $\phi$.
	Let~$\vd_\phi((i,j),(i',j'))$ be the number of rows $i''$ such that $i$ is hit by $\phi$ and $i < i'' \le i'$. Similarly, let $\hd_\phi((i,j),(i',j'))$ be the number of columns $j''$ such that $j$ is hit by $\phi$ and $j < j'' \le j'$. For~$A \subseteq E(M)$, let $\width_\phi(A) = \max_{x,y \in A} \hd_\phi(x,y)$, and $\height_\phi(A) = \max_{x,y \in A} \vd_\phi(x,y)$.
	\begin{observation}\label{p:phi_obs}
		Let $\phi$ be an embedding of $P$ into $M$, let $x, y \in E(P)$, and let $\phi(x), \phi(y) \in A \subseteq E(M)$. Then
		\begin{align*}
			& \hd(x,y) = \hd_\phi( \phi(x), \phi(y) ) \le \hd(\phi(x), \phi(y) ) \le \width(A); \text{ and}\\
			& \vd(x,y) = \vd_\phi( \phi(x), \phi(y) ) \le \vd(\phi(x), \phi(y) ) \le \height(A).\tag*{\qed}
		\end{align*}
	\end{observation}
	
	\subsection{Constructing witnesses}\label{sec:construct_examples}
	
	\begin{figure}[t]
		\centering
		\newcommand{\doubleNode}[3]{\node[#2] at (#1) {}; \node[#3] at (#1) {};}
\newcommand{\capDist}{0.5}
\newcommand{\boxDist}{4}
\newcommand{\vboxDist}{2}
\newcommand{\capnode}[2]{\node[left] at (-\capDist,#2) {\emph{#1}};}
\begin{tikzpicture}[
		scale=0.28,
		circlePoint/.style={circle, inner sep=1.2pt},
		rectPoint/.style={rectangle, inner sep=1.6pt},
		diamPoint/.style={diamond, inner sep=1.1pt},
		starPoint/.style={star, star points=5, star point ratio=2, inner sep=1pt},
		colorA/.style={blue},
		colorB/.style={orange!80!red},
		colorC/.style={green!60!black},
		colorD/.style={cyan},
		point/.style={circlePoint, fill},
		expRowPoint/.style={circlePoint, draw},
		pointA/.style={circlePoint, fill, colorA},
		expRowPointA/.style={circlePoint, draw, colorA},
		pointB/.style={rectPoint, fill, colorB},
		expRowPointB/.style={rectPoint, draw, colorB},
		pointC/.style={diamPoint, fill, colorC},
		expRowPointC/.style={diamPoint, draw, colorC},
		pointD/.style={starPoint, fill, colorD},
		expRowPointD/.style={starPoint, draw, colorD},
		pointH/.style={circle, draw, red, inner sep=2pt},
		sepline/.style={dashed, gray}
	]
	\begin{scope}[shift={(0,0)}]
		\begin{scope}[shift={(0,0)}]
			\capnode{a}{2.5}
			\draw (0,0) rectangle (5,5);
			\node[point] at (1,3) {};
			\node[point] at (2,1) {};
			\node[point] at (3,4) {};
			\node[point] at (4,2) {};
		\end{scope}
		\begin{scope}[shift={(5+\boxDist,0)}]
			\capnode{b}{2.5}
			\draw (0,0) rectangle (8,5);
			
			\begin{scope}[shift={(3,0)}]
				\node[point] at (2,1) {};
				\node[point] at (3,4) {};
				\node[point] at (4,2) {};
			\end{scope}
			\foreach \i in {0,1,...,3}{
				\node[expRowPoint] at (\i+1, 3) {};
			}
		\end{scope}
		\begin{scope}[shift={(13+2*\boxDist,0)}]
			\capnode{c}{2.5}
			\draw (0,0) rectangle (8,5);
			
			\begin{scope}[shift={(0,0)}]
				\node[point] at (1,3) {};
				\node[point] at (2,1) {};
				\node[point] at (3,4) {};
			\end{scope}
			\foreach \i in {4,...,7}{
				\node[expRowPoint] at (\i, 2) {};
			}
		\end{scope}
		\begin{scope}[shift={(21+3*\boxDist,-0.5)}]
			\capnode{d}{3}
			\draw (0,0) rectangle (11,6);
			
			\begin{scope}[shift={(4,0)}]
				\node[pointA] at (2,1) {};
				\node[pointA] at (3,4) {};
				\node[pointA] at (4,2) {};
			\end{scope}
			\foreach \i in {1, ..., 5}{
				\node[expRowPointA] at (\i, 3) {};
			}
			
			\begin{scope}[shift={(2,1)}]
				\node[pointB] at (1,3) {};
				\node[pointB] at (2,1) {};
				\node[pointB] at (3,4) {};
			\end{scope}
			\foreach \i in {6,...,10}{
				\node[expRowPointB] at (\i, 3) {};
			}
		\end{scope}
	\end{scope}
	\begin{scope}[shift={(0,-8-\vboxDist)}]
		\begin{scope}[shift={(0,0)}]
			\draw (0,0) rectangle (7,7);
			\capnode{e}{3.5}
			
			\node[point] at (1,4) {};
			\node[point] at (2,6) {};
			\node[point] at (3,2) {};
			\node[point] at (4,5) {};
			\node[point] at (5,1) {};
			\node[point] at (6,3) {};
		\end{scope}
		\begin{scope}[shift={(7+\boxDist,-0.5)}]
			\draw (0,0) rectangle (11,8);
			\capnode{f}{4}
			
			\doubleNode{1,5}{point}{pointH}
			\doubleNode{2,7}{point}{pointH}
			\node [point] at (3,3) {};
			\node [point] at (4,6) {};
			\doubleNode{5,2}{point}{pointH}
			\doubleNode{6,6}{point}{pointH}
			\node [point] at (7,2) {};
			\node [point] at (8,5) {};
			\doubleNode{9,1}{point}{pointH}
			\doubleNode{10,3}{point}{pointH}
			\foreach \i in {1, ..., 10}{
				\node[expRowPoint] at (\i, 4) {};
			}
		\end{scope}
		\begin{scope}[shift={(18+2*\boxDist,-0.5)}]
			\draw (0,0) rectangle (11,8);
			\capnode{g}{4}
			
			\node [pointB] at (1,5) {};
			\node [pointB] at (2,7) {};
			\node [pointB] at (3,3) {};
			\node [pointB] at (4,6) {};
			\node [pointB] at (5,2) {};
			\foreach \i in {6, ..., 10}{
				\node[expRowPointB] at (\i, 4) {};
			}
			
			\node [pointA] (moved) at (6,7) {};
			\node [pointA] at (7,2) {};
			\node [pointA] at (8,5) {};
			\node [pointA] at (9,1) {};
			\node [pointA] at (10,3) {};
			\foreach \i in {1, ..., 5}{
				\node[expRowPointA] at (\i, 4) {};
			}
						
			\draw [-latex,red] (6,6) -- (moved);
		\end{scope}
	\end{scope}
	
	\begin{scope}[shift={(0,-15.5-2*\vboxDist)}]
		\begin{scope}[shift={(0,0)}]
			\draw (0,0) rectangle (6,6);
			\capnode{h}{3}
			
			\node[point] at (1,4) {};
			\node[point] at (2,1) {};
			\node[point] at (3,3) {}; \node[below] at (3,3) {\small $x$};
			\node[point] at (4,5) {};
			\node[point] at (5,2) {};
		\end{scope}
		
		\begin{scope}[shift={(6+\boxDist,0)}]
			\draw (0,0) rectangle (9,6);
			\capnode{i}{3}
			
			\node[point] at (1,4) {};
			\node[point] at (2,1) {};
			\foreach \i in {3, ..., 6}{
				\node[expRowPoint] at (\i, 3) {};
			}
			\node[point] at (7,5) {};
			\node[point] at (8,2) {};
		\end{scope}
		
		\begin{scope}[shift={(15+2*\boxDist,-1)}]
			\draw(0,0) rectangle (13,8);
			\capnode{j}{4}
%			\draw[sepline] (2.5,0) -- (2.5,8);
%			\draw[sepline] (6.5,0) -- (6.5,8);
%			\draw[sepline] (10.5,0) -- (10.5,8);
			
			\node[pointC] at (1,5) {};
			\node[pointC] at (2,2) {};
			\foreach \i in {3, ..., 10}{
				\node[expRowPointC] at (\i, 4) {};
			}
			\node[pointC] at (11,6) {};
			\node[pointC] at (12,3) {};
			
			\node[pointB] at (7,6) {};
			\node[pointB] at (8,3) {};
			\node[pointB] at (9,5) {};
			\node[pointB] at (10,7) {};
			\foreach \i in {11,12}{
				\node[expRowPointB] at (\i, 4) {};
			}
			
			\foreach \i in {1,2}{
				\node[expRowPointA] at (\i, 4) {};
			}
			\node[pointA] at (3,1) {};
			\node[pointA] at (4,3) {};
			\node[pointA] at (5,5) {};
			\node[pointA] at (6,2) {};
		\end{scope}
	\end{scope}
	
	\begin{scope}[shift={(0,-26.5-3*\vboxDist)}]
		\begin{scope}
			\draw(0,0) rectangle (21,10);
			\capnode{k}{5}
			
			\node[pointC] at (1,4) {};
			\node[pointC] at (2,6) {};
			\node[pointC] at (3,2) {};
			\node[pointA] at (4,7) {};
			\node[pointA] at (5,3) {};
			\node[pointA] at (6,6) {};
			\node[pointA] at (7,2) {};
			\node[pointA] at (8,4) {};
			\node[pointD] at (9,7) {};
			\node[pointD] at (10,9) {};
			\node[pointC] at (11,1) {};
			\node[pointC] at (12,3) {};
			\node[pointB] at (13,6) {};
			\node[pointB] at (14,8) {};
			\node[pointB] at (15,4) {};
			\node[pointB] at (16,7) {};
			\node[pointB] at (17,3) {};
			\node[pointD] at (18,8) {};
			\node[pointD] at (19,4) {};
			\node[pointD] at (20,6) {};
			
			\foreach \i in {1,...,3}{ \node[expRowPointA] at (\i, 5) {}; }
			\foreach \i in {4,...,10}{ \node[expRowPointC] at (\i, 5) {}; }
			\foreach \i in {11,...,17}{ \node[expRowPointD] at (\i, 5) {}; }
			\foreach \i in {18,...,20}{ \node[expRowPointB] at (\i, 5) {}; }
		\end{scope}
	\end{scope}
\end{tikzpicture}
		\caption{Constructing witnesses as described in \cref{sec:construct_examples}. Empty dots indicate positions that complete the pattern, i.e., (partial) expandable rows. Colors/shapes of dots (1-entries) in \emph{d, g, j}, and \emph{k} indicate different parts of the construction. Red circles in \emph{f} indicate the occurrence of the pattern.}\label{fig:cons_ex}
	\end{figure}
	
	In this subsection, we describe intuitively how our various constructions work. Recall that a vertical witness for a matrix $P$ avoids $P$ and has a $P$-expandable row. As we will see in the following, there are many different ways of constructing a matrix with a $P$-expandable row; ensuring that it also avoids $P$ is the hard part.
	
	Suppose we have a permutation matrix $P$. Place a copy of $P$ in the middle of a large matrix, and remove the leftmost 1-entry $\ell$ of the copy (\cref{fig:cons_ex} \emph{a, b}). Then, adding a 1-entry at the position of $\ell$ or in the same row to the left of $\ell$ will complete an occurrence of $P$. Thus, we have constructed the ``left part'' of an expandable row. We can similarly use a copy of $P$ missing the rightmost 1-entry $r$ to get the ``right part'' of an expandable row (\cref{fig:cons_ex} \emph{c}). If we now place the copy without $\ell$ (call it $L$) to the right of the copy without $r$ (call it $R$), we can create a whole expandable row (\cref{fig:cons_ex} \emph{d}).
	
	It turns out that the resulting matrix avoids $P$ for a large class of patterns, in particular for patterns with a spanning oscillation of length 4. We prove this in \cref{sec:4-trav}.
	
	The smallest example where the construction does contain $P$ is shown in \cref{fig:cons_ex}~\emph{e,~f}.
	However, observe that changing the vertical position of a 1-entry preserves the expandable row as long as the vertical order within both $L$ and $R$ is maintained. Thus, we can try to vertically ``stretch'' $L$ and/or $R$ to make the matrix avoid $P$. In the given example, moving the top entry of~$L$ up one row suffices (\cref{fig:cons_ex}~\emph{g}).
	
	When stretching does not help, another option is to construct larger matrices in the following way. Consider some 1-entry $x$ of~$P$ that is not the leftmost or rightmost 1-entry. Place a copy of~$P$ into a large matrix $M$, remove $x$, and then move all 1-entries to the left of $x$ further to the left. This creates a ``middle part'' of an expandable row (\cref{fig:cons_ex}~\emph{h, i}). Call that modified matrix $P_x$. Arranging $L$, $R$, and $P_x$ as in \cref{fig:cons_ex}~\emph{j} completes an expandable row. In \cref{sec:surround} we use this construction, except that $P_x$ is additionally stretched in a certain way.
	
	More generally, we can obtain an expandable row by interleaving copies of $L$, $R$ \linebreak and $P_{x_1},\allowbreak P_{x_2}, \dots$ for several 1-entries $x_1, x_2, \dots$ of $P$. \Cref{fig:cons_ex}~\emph{k} shows another witness for the matrix in \cref{fig:cons_ex}~\emph{e} constructed in this way. Observe that each partial copy of $P$ takes care of a certain part of the expandable row. This idea forms the basis for the construction in \cref{sec:complicated}, where we use a number of different partial copies $P_{x_i}$ depending on the pattern $P$.
	
	\section{Spanning oscillations of length 4}\label{sec:4-trav}
	
	In this section, we show \cref{p:4-trav-gen}, which immediately implies \cref{p:4-osc}.
	
	\restateFourTravGen*
	
	Let $\fP$ denote the class of patterns defined in \cref{p:4-trav-gen}. Note that $\fP$ is closed under transposition. Thus, by \cref{p:vert_wit_suff}, it is sufficient to prove that each $P \in \fP$ has a vertical witness.
	
	Let $\fP'$ be the subset of patterns $P \in \fP$ where the unique leftmost 1-entry $\ell$ of $P$ is above the unique rightmost 1-entry $r$ of $P$. It is easy to see that each $P \in \fP'$ has the following form, where the boxes contain arbitrarily many 1-entries:
	\begin{center}
		\begin{tikzpicture}[scale=0.5]
			\tNamedPoint{1,3}{left}{$\ell$}
			\tNamedPoint{3,4}{above}{$t$}
			\tNamedPoint{2,1}{below}{$b$}
			\tNamedPoint{4,2}{right}{$r$}
			
			\newcommand{\tMargin}{0.2}
			\newcommand{\tSmallBox}[2]{\draw (#1+\tMargin,#2+\tMargin) rectangle (#1+1-\tMargin,#2+1-\tMargin);}
			\tSmallBox{1}{1}
			\tSmallBox{1}{2}
			\tSmallBox{1}{3}
			\tSmallBox{2}{1}
			\tSmallBox{2}{2}
			\tSmallBox{2}{3}
			\tSmallBox{3}{1}
			\tSmallBox{3}{2}
			\tSmallBox{3}{3}
		\end{tikzpicture}
	\end{center}
	
	Since for each $P \in \fP \setminus \fP'$, we have $\rev(P) \in \fP'$, \cref{p:transform_effects} implies that it is sufficient to prove that each $P \in \fP'$ has a vertical witness.
	
	\begin{lemma}\label{p:4_trav_wit}
		Each $P \in \fP'$ has a vertical witness.
	\end{lemma}
	\begin{proof}
		Let $P \in \fP'$ be a $k_1 \times k_2$ pattern, let $\ell = (i,j)$ be the unique leftmost 1-entry in $P$, and let $r = (i',j')$ be the unique rightmost 1-entry in $P$. Note that $i < i'$.
		
		We essentially use the construction shown in \cref{fig:cons_ex} \emph{d}.
		Let $P_\rL$ and $P_\rR$ be the sub\-ma\-tri\-ces of $P$ obtained by removing the rightmost, resp. leftmost, column. Note that in $P_\rL$, the $i'$-th row is empty, and in $P_\rR$, the $i$-th row is empty. As described in \cref{sec:construct_examples}, the idea is to place a copy of $P_\rL$ to the left of $P_\rR$, so that the two empty rows coincide. More formally, obtain $L$ from~$P_\rL$ by appending $i'-i > 0$ rows (at the bottom), obtain $R$ from $P_\rR$ by prepending $i' - i > 0$ rows (at the top), and define $S(P)$ as the horizontal concatenation $(L, R)$. Note that $S(P)$ is \linebreak a $(k_1 + i'-i) \times (2k_2-2)$ matrix, and that the $i'$-th row of $S(P)$ is empty.
		In the following, we use $L$ and $R$ interchangeably with the corresponding subsets of $E(S(P))$.
		
		We claim that the $i'$-th row is $P$-expandable. Indeed, adding a 1-entry in the $i'$-th row in the first $k-1$ columns (to the left of $R$) completes an occurrence of $P$ with $R$, and adding a 1-entry in the last $k-1$ columns (to the right of $L$) completes an occurrence of $P$ with $L$.
		
		It remains to show that $S(P)$ avoids $P$. Suppose $S(P)$ contains $P$, so there is an embedding~$\phi$ of $P$ into $S(P)$. Let $t,b \in E(P)$ be the unique topmost, respectively bottommost, 1-entry in~$P$.
		
		Suppose first that $\phi(b) \in L$. Since $\height(L) = \vd(t,b) = k-1$, and the $i'$-th row of $P$ is empty, we have $\height_\phi(L) < \vd(t,b)$. This implies that $\phi(t)$ is above $L$. But $S(P)$ has no 1-entries above $L$, a contradiction.
		
		Otherwise, $\phi(b) \in R$. Since $t$ is to the right of $b$, this implies that $\phi(t) \in R$. But a similar argument as above shows that $\height_\phi(R) < \vd(t,b)$, a contradiction.
	\end{proof}
	
	\section{Spanning oscillations starting with \texorpdfstring{$t$}{t}}\label{sec:surround}
	
	In this section, we prove:
	\restateInvTravWit*
	
	In \Cref{sec:sur-wit}, we present a construction of (possible) witnesses, which we use for the case~$m = 5$ in \Cref{sec:sur-5}, and for the case $m > 5$ in \Cref{sec:sur-long}.
	
	\subsection{Witness construction}\label{sec:sur-wit}
	
	Let $P$ be an $k \times k$ permutation matrix such that $\ell = \ell_P$ is above $r = r_P$, and  \linebreak let ${q = (i_q, j_q) \in E(P)}$, such that $q$ is above $\ell$. We first construct a matrix $S'(P,q)$ with a \linebreak $P$-expandable row, in the way shown in \cref{fig:cons_ex} \emph{j}. Then, we modify $S'(P,q)$ to obtain the matrix $S(P,q)$, which retains the expandable row and will be shown to avoid $P$ if $P$ has a wide spanning oscillation $(t_P, \ell_P, x_3, x_4, \dots, x_m)$ with $m \ge 5$ and we choose $q = x_3$.
	
	Let $P_\rR$ ($P_\rL$) be the submatrix of $P$ obtained by removing the leftmost (rightmost) column. Both $P_\rR$ and $P_\rL$ have an empty row. To start the construction of $S'(P,q)$, we place a copy of~$P_\rR$ to the \emph{left} of a copy of $P_\rL$, such that the two copies do not intersect, and the empty rows are aligned. We denote the copy of $P_\rR$ in the construction with $R$ and the copy of $P_\rL$ with $L$. Note that, compared to the construction in \cref{sec:4-trav}, $L$ and $R$ switch places.
	
	Let $P_\rL'$ consist of all columns of $P$ to the left of $q$, and let $P_\rR'$ consist of all columns of $P$ to the right of $q$. To finish the construction of $S'(P,q)$, we place a copy of $P_\rL'$ to the left of $R$ and a copy of $P_\rR'$ to the right of $L$, such that the empty $i_q$-th rows of $P_\rL'$ and $P_\rR'$ are aligned with the empty row in $R$ and $L$. Denote the copies of $P_\rL'$ and $P_\rR'$ as $L'$ and $R'$ and let $P' = L' \cup R'$.
	
	Clearly, the empty row in $S'(P,q)$ is expandable: Adding a 1-entry to the left of $R$ will complete the partial occurrence $R$ of $P$, adding a 1-entry to the right of $L$ will complete $L$, and adding a 1-entry within $R$ or $L$ will complete $P'$.
	
	We modify $S'(P,q)$ to obtain $S(P,q)$ as follows.\footnote{This idea comes from Geneson's construction.~\cite{Geneson2021}} Let $B$ be the set of entries in $P' = L' \cup R'$ that are below the leftmost 1-entry in $P'$ (the copy of $\ell$ in $P'$). Move $B$ down by a fixed number of rows, such that each 1-entry in $B$ is lower than all 1-entries in $R \cup L$. Clearly, the expandable row stays expandable after this change.
	
	\Cref{fig:surround-constr} sketches the constructions. In the following sections, we denote the 1-entries in~$S(P,q)$ as follows. If $x$ is a 1-entry in $P$, then let $x^\rR$ be the copy of $x$ in $R$, let $x^\rL$ be the copy of $x$ in $L$, and let $x'$ be the copy of $x$ in $P'$. For subsets $X \subseteq E(P)$, we use $X^\rR$, $X^\rL$ and $X'$ similarly.
	
	\begin{figure}
		\centering
		\begin{tikzpicture}[scale=0.3]
			\draw (0,0) rectangle (4,4);
			\tNamedPoint{2,3}{right}{$q$}
			\tNamedPoint{0,2}{left}{$\ell$}
			\tNamedPoint{4,1}{right}{$r$}
			\node[below] at (2,0) {$P$};
		\end{tikzpicture}
		\hspace{9mm}
		\begin{tikzpicture}[scale=0.3,rotate=180]
			\draw (0,-3) rectangle (4,0);
			\draw (0,1) rectangle (4,2);
			\node[above] at (2,-3) {$L$};
			
			\draw (5,-2) rectangle (9,0);
			\draw (5,1) rectangle (9,3);
			\node[above] at (7,-2) {$R$};
			
			\draw (-3,-1) rectangle (-1,0);
			\draw (-3,1) rectangle (-1,4);
			\node[above] at (-2,-1) {$R'$};
			
			\draw (10,-1) rectangle (12,0);
			\draw (10,1) rectangle (12,4);
			\node[above] at (11,-1) {$L'$};
			\tNamedPoint{12,2}{left}{$\ell'$}
			
			\node at (4.5,8) {$S'(P,q)$};
		\end{tikzpicture}
		\hspace{9mm}
		\begin{tikzpicture}[scale=0.3,rotate=180]
			\draw (0,-3) rectangle (4,0);
			\draw (0,1) rectangle (4,2);
			\node[above] at (2,-3) {$L$};
			
			\draw (5,-2) rectangle (9,0);
			\draw (5,1) rectangle (9,3);
			\node[above] at (7,-2) {$R$};
			
			\draw (-3,-1) rectangle (-1,0);
			\draw (-3,1) rectangle (-1,2);
			\draw (-3,4) rectangle (-1,6);
			\node[above] at (-2,-1) {$R'$};
			
			\draw (10,-1) rectangle (12,0);
			\draw (10,1) rectangle (12,2);
			\draw (10,4) rectangle (12,6);
			\node[above] at (11,-1) {$L'$};
			\node[left] at (12,0.5) {$L'_1$};
			\node[left] at (12,5) {$L'_2$};
			\tNamedPoint{12,2}{left}{$\ell'$}
			
			\node at (4.5,8) {$S(P,q)$};
		\end{tikzpicture}
		\caption{A sketch of $P$ and the two witness constructions $S'(P,q)$ and $S(P,q)$.}\label{fig:surround-constr}
	\end{figure}
	
	We now show a property of $S(P,q)$ that is useful in both of the following subsections.
	
	\begin{lemma}\label{p:surr-t-b-loc}
		Let $P$ be a $k \times k$ permutation matrix and $q \in E(P)$ such that $q \ltv \ell_P \ltv r$ and~$t_P$ is to the left of $b_P$. If $\phi$ is an embedding of $P$ into $S(P,q)$, then $\phi(t_P) \notin L'$ and $\phi(b_P) \in R'$.
	\end{lemma}
	\begin{proof}
		We write $\ell, t, b, r$ for $\ell_P, t_P, b_P, r_P$. Let $L'_2$ denote the portion of $L'$ below $\ell'$, and \linebreak let~$L'_1 = L' \setminus L'_2$.
		
		We first show that $\phi(t) \notin L'$. Suppose $\phi(t) \in L'$. Then also $\phi(\ell) \in L'$. Since \linebreak $\height(L'_2) < \vd(\ell, b)$, and there are no nonempty rows below $L'_2$, we know that $\phi(\ell) \notin L'_2$, and therefore $\phi(t), \phi(\ell) \in L'_1$. But $\height_\phi(L'_1) \le \vd(t,\ell)-1$, a contradiction.
		
		$\phi(t) \notin L'$ already shows that $\phi(b) \notin L'$, since $b$ is to the right of $t$. It remains to show that~$\phi(b) \notin R \cup L$. First, suppose that $\phi(b) \in L$. Then there are at most $k-2$ nonempty rows above $\phi(b)$, but $\vd(t,b) = k-1$, a contradiction.
		
		Second, suppose that $\phi(b) \in R$. Then also $\phi(t) \in R$, because $t$ is to the left of $b$ \linebreak and~$\phi(t) \notin L'$. But $\height_\phi(R) \le \vd(t,b)-1$, a contradiction.
	\end{proof}
	
	\subsection{Spanning oscillations of length five}\label{sec:sur-5}
	
	\begin{lemma}\label{p:corner}
		Let $P$ be a permutation matrix and $X = (t_P, x_2, x_3, x_4, x_5)$ be a a spanning oscillation of $P$. Then $S(P,x_3)$ avoids $P$.
	\end{lemma}
	\begin{proof}
		Let $q = x_3$, and note that $x_2 = \ell_P$ and $x_4 = b_P$, so $q$ is above $\ell_P$ and to the right \linebreak of $b_P$.
		Suppose $\phi$ is an embedding of $P$ into $S(P,q)$. By \Cref{p:surr-t-b-loc}, $\phi(b_P) \in R'$. \linebreak But $\width(R') = \hd(q,r_P)-1 < \hd(b_P,r_P)$, a contradiction.
	\end{proof}

	\subsection{Longer spanning oscillations}\label{sec:sur-long}
	
	We now consider the case where $P$ has a wide spanning oscillation $(t_P, x_2, \dots, x_m)$ of length greater than five.
	We first prove a useful property of long spanning oscillations.
	
	\begin{lemma}\label{p:prop-type-2}
		Let $P$ be a permutation matrix and $X = (t_P, x_2, \dots, x_m)$ be a spanning oscillation of $P$ with $m \ge 6$. Then, removing $t = t_P$, the columns to the left of $t$, and the rows above $x_3$ (as well as all newly created rows or columns) does not make $P$ decomposable.
	\end{lemma}
	\begin{proof}
		Suppose it does, and let $P_0$ be the resulting decomposable pattern. Since $x_3$ is the highest 1-entry in $P_0$ (slightly abusing notation), and $x_3$ is above $r = r_P$ and to the left of $b = b_P$, we know that $P_0$ has the form $\left(\begin{smallmatrix}A&\zeromat\\\zeromat&B\end{smallmatrix}\right)$, where $x_3$ lies in $A$ and $r$, $b$ lie in $B$.
		This means that $x_4$ lies in~$A$, since $t \lth x_4 \lth x_3$. Let $P_1$ be the matrix obtained from $P_0$ by further removing all columns to the right of $x_4$. Clearly, $P_1$ is decomposable, but $(x_3, x_4, \dots, x_m)$ is a spanning oscillation of~$P_1$, a contradiction.
	\end{proof}
	
	We are now ready to prove the main result of this subsection.
	
	\begin{figure}[htbp]
		\centering
		\begin{tikzpicture}[scale=0.3,rotate=180]
			\draw (0,3) rectangle (2,5);
			\draw (2,5) -- (4,5) -- (4,3);
			\draw (2,1) -- (3,1) -- (3,5);
			\draw (2,0) rectangle (4,3);
			\draw[dashed] (2,1) -- (0,1) -- (0,3);
			\node at (1,4) {\footnotesize $A$};
			\node at (2.5,4) {\footnotesize $B$};
			\node at (2.5,2) {\footnotesize $C$};
			\node at (-1,1) {\small $P_0$};
			\tNamedPoint{0,4}{right}{$r$}
			\tNamedPoint{1,5}{below}{$b$}
			\tNamedPoint{2,1}{above right}{$q$}
			\tNamedPoint{3,0}{above}{$t$}
			\tNamedPoint{4,2}{left}{$\ell$}
		\end{tikzpicture}
		\hspace{10mm}
		\begin{tikzpicture}[scale=0.3,rotate=180]
			\draw (0,-1) -- (0,0) -- (4,0) -- (4,-4) -- (2,-4) -- (2,-1) -- (0,-1);
			\draw (0,1) rectangle (4,2);
			\node[right] at (2,-2.5) {$L$};
			
			\draw (7,-2) rectangle (9,0);
			\draw (5,2) -- (5,4) -- (9,4) -- (9,1) -- (7,1) -- (7,2) -- (5,2);
			\node[right] at (7,0.5) {$R$};
			
			\draw (-3,6) rectangle (-1,8);
			\node[above] at (-2,6) {$R'$};
			
			\draw (10,-1) rectangle (12,0);
			\node[left] at (12,0) {$L'_1$};
			\draw (10,1) rectangle (12,2);
			\draw (10,5) rectangle (12,8);
			\node[left] at (12,6) {$L'_2$};
			
			\tNamedPoint{11,-1}{above}{$t'$}
			\tNamedPoint{12,2}{left}{$\ell'$}
		\end{tikzpicture}
		\caption{$P$ and $S(P,q)$ in the case of \Cref{p:mid}.}
		\label{fig:P-between-tbr}
	\end{figure}
	
	\begin{lemma}\label{p:mid}
		Let $X = (t_P, x_2, \dots, x_m)$ be a wide spanning oscillation of $P$ with $m \ge 6$. Then~$P$ has a vertical witness.
	\end{lemma}
	\begin{proof}
		We write $\ell, t, b, r$ for $\ell_P, t_P, b_P, r_P$ in the following.
		Let $q = x_3$, and let $P_0$ be the set of 1-entries of $P$ that are to the right of $t$ and not above $q$. By \Cref{p:prop-type-2}, $P_0$ does not correspond to a decomposable pattern. Let $A$ denote the set of 1-entries to the right of $q$. Note that $b, r \in A$, and, by wideness of $X$, all 1-entries in $A$ are below $\ell$. Let $x$ be the highest 1-entry in $A$, and let~$B$ be the set of 1-entries below $x$, to the left of $q$ and to the right of $t$. Then $B \neq \varnothing$, otherwise~$P_0$ would be decomposable. Finally, $C = P_0 \setminus (A \cup B)$ consists of the 1-entries to the right of $t$, not above $q$, and above $x$. \Cref{fig:P-between-tbr} shows a sketch of $P$ and $S(P,X)$. Note that $A' = R'$ (recall that $A'$ denotes the copy of $A$ in $P'$).
		
		Suppose $\phi$ is an embedding of $P$ into $S(P,q)$. By \cref{p:surr-t-b-loc}, $\phi(b) \in R'$ and $\phi(t) \notin L'$. Since all 1-entries in $B$ are to the right of $t$, this implies $\phi(y) \notin L'$ for each $y \in B$. Moreover, $\width(R') = \hd(q, r) - 1 < \hd(y,r)$ for each $y \in B$, so we have $\phi(B) \subseteq L \cup R$.
		
		Let $L'_2$ denote the portion of $L'$ below $\ell'$ and let $L'_1 = L' \setminus L'_2$. Note that $L'_2$ is below all 1-entries in $L \cup R$. Since all 1-entries in $C$ are above all 1-entries in $B$, and all 1-entries in $A$ are to the right of all 1-entries in $B$, we have $\phi(P_0) = \phi(A \cup B \cup C) \subseteq L'_1 \cup L \cup R \cup R'$. Since $R' = A'$, all 1-entries in $R'$ are to the right and below all 1-entries in $L'_1 \cup L \cup R$, so~$L'_1 \cup L \cup R \cup R'$ can be decomposed into the two blocks $L'_1 \cup L \cup R$ and $R'$. Further, $\phi(b) \in R'$ by \cref{p:surr-t-b-loc}, and since $\height(R') < \vd(q,b)$, we have $\phi(q) \notin R'$. This means that $P_0$ is decomposable, a contradiction.
	\end{proof}

	\section{Even-length spanning oscillations starting with \texorpdfstring{$\ell$}{l}}\label{sec:complicated}
	
	In this section, we prove:
	\restateTravWit*
	
	For our witness construction to work, we need to define a substructure that generalizes (tall) spanning oscillations of even length that start with $\ell_P$. We call that substructure a \emph{traversal}. Defining our witness construction for traversals instead of spanning oscillations allows us to make a maximality assumption that is required later in the proof.
	
	\subsection{Traversals}
	
	Let $P$ be a permutation matrix and let $m \ge 4$. A \emph{traversal} of $P$ is a sequence $X$ of distinct 1-entries $x_1, x_2, \dots, x_m$ such that
	\begin{enumerate}[(i)]
		\item $x_1 = \ell_P$, $x_2 = t_P$, $x_{m-1} = b_P$, $x_m = r_P$;\label{prop:trav-ltbr}
		\item $x_1 \lth x_3 \lth x_2 \lth x_5 \lth x_4 \lth \dots \lth x_{m-1} \lth x_{m-2} \lth x_m$;\label{prop:trav-hor}
		\item $\ell_P \ltv x_4 \ltv x_6 \ltv \dots \ltv x_m$;\label{prop:trav-upper}
		\item $x_3 \ltv x_5 \ltv \dots \ltv x_{m-3} \ltv r_P$; and\label{prop:trav-lower}
		\item $x_s$ is below $x_{s+1}$ for each odd $s \in [m-1]$.\label{prop:trav-vert}
		\setcounter{resumeEnum}{\value{enumi}}
	\end{enumerate}

	\begin{figure}[tbp]
		\centering
		\begin{tikzpicture}[scale=0.3]
			\draw (1,-1) -- (4,-1) -- (4,-3) -- (6,-3) -- (6,-4) -- (8,-4) -- (8,-8) -- (5,-8) -- (5,-6) -- (3,-6) -- (3,-5) -- (1,-5) -- (1,-1);
			\tNamedPoint{1,-2}{left}{$x_1$}
			\tNamedPoint{3,-1}{above}{$x_2$}
			\tNamedPoint{2,-5}{below}{$x_3$}
			\tNamedPoint{5,-3}{above}{$x_4$}
			\tNamedPoint{4,-6}{below}{$x_5$}
			\tNamedPoint{7,-4}{above}{$x_6$}
			\tNamedPoint{6,-8}{below}{$x_7$}
			\tNamedPoint{8,-7}{right}{$x_8$}
		\end{tikzpicture}
		\caption{A tall traversal. The solid lines indicate the boundary of possible positions for other 1-entries.}\label{fig:trav}
	\end{figure}

	Intuitively, property (\ref{prop:trav-hor}) enforces the same horizontal order on the 1-entries as an even-length spanning oscillation.
	Vertically, however, we are allowed to arrange the 1-entries more freely. There are still \emph{upper} (even) and \emph{lower} (odd) 1-entries as in \cref{p:spanosc-upper-lower} (this is implied by (\ref{prop:trav-upper}), (\ref{prop:trav-lower}), (\ref{prop:trav-vert})), and we keep the order within the upper, resp. lower, 1-entries with (\ref{prop:trav-upper}), (\ref{prop:trav-lower}). But we drop the condition that $x_i$ is above $x_j$ for each odd $i \le m-3$ and even $j \ge i+3$. This means that we are allowed to ``move'' some upper 1-entries upwards, and some lower 1-entries downwards, as long as the vertical order among upper (lower) 1-entries is kept intact. (\ref{prop:trav-upper}), (\ref{prop:trav-lower}) additionally ensure that we cannot move any 1-entries above $\ell_P$ or below $r_P$. \Cref{fig:trav} shows the shortest traversal that is not an oscillation.
	
	We say a traversal $(x_1, x_2, \dots, x_m)$ is \emph{tall} if it satisfies the following two properties for each even $2 \le i \le m-2$.
	\begin{enumerate}[(i)]
		\setcounter{enumi}{\value{resumeEnum}}
		\item $P$ has no 1-entry that is below $x_{i+1}$ and to the left of $x_i$.\label{prop:trav-tall-below}
		\item $P$ has no 1-entry that is above $x_i$ and to the right of $x_{i+1}$.\label{prop:trav-tall-above}
	\end{enumerate}
	
	\begin{observation}\label{p:osc-is-trav}
		Each tall spanning oscillation of even length that starts with $\ell$ is a tall traversal.\qed
	\end{observation}
	
	\subsection{Maximality assumption}
	Let $P$ be a permutation matrix with a tall traversal $X$. We can assume that $X$ is maximal in the sense that no tall traversal of $P$ has $X$ as a proper subsequence. We now show that such a \emph{maximally tall} traversal also cannot be extended to a larger non-tall traversal in the following sense. Call a traversal $(x_1, x_2, \dots, x_m)$ \emph{extendable} if there is an odd $s$ with $5 \le s \le m-5$, and two 1-entries $y_1, y_2$ in $P$ such that $(x_1, x_2, \dots, x_s, y_1, y_2, x_{s+1}, \dots, x_m)$ is a traversal of $P$.
	
	\begin{lemma}\label{p:max-tall}
		Let $X = (x_1, x_2, \dots, x_m)$ be a maximally tall traversal of the permutation matrix~$P$. Then $X$ is non-extendable.
	\end{lemma}
	\begin{proof}
		Suppose $X$ is extendable. Then there exists an odd $s$ with $5 \le s \le m-5$ and 1-entries $y_1, y_2 \in E(P)$ such that $Y = (x_1, x_2, \dots, x_s, y_1, y_2, x_{s+1}, \dots, x_m)$ is a traversal of $P$. We show that then $P$ has a tall traversal of length $m+2$ with $X$ as a subsequence. This contradicts our assumption that $X$ is maximally tall.
		
		Note that property (\ref{prop:trav-vert}) of $X$ implies that $x_{s+1}$ is above $x_s$. Further using properties (\ref{prop:trav-hor}), (\ref{prop:trav-upper}), (\ref{prop:trav-lower}) of $Y$, it follows that the relative positions of $x_{s-1}$, $x_s$, $y_1$, $y_2$, $x_{s+1}$, and $x_{s+2}$ are fixed as shown in \cref{fig:max-tall-help}.
		
		\begin{figure}
			\centering\small
			
			\newcommand{\zshade}[4]{\draw[fill, {black!10}] (#1,#2) rectangle (#3,#4);}
			\newcommand{\zshadeb}[4]{\draw (#1,#2) -- (#3,#2);\draw (#1,#2) -- (#1,#4);}
			\tikzset{
				common-max-tall/.pic={code={
						\zshade{0}{6}{6}{8}
						\zshade{4}{4}{6}{8}
						\zshadeb{0}{6}{6}{8}
						\zshadeb{4}{4}{6}{8}
						
						\zshade{1}{3}{-1}{-1}
						\zshade{5}{1}{-1}{-1}
						\zshadeb{1}{3}{-1}{-1}
						\zshadeb{5}{1}{-1}{-1}
						
						\tNamedPoint{1,6}{below}{$x_{s-1}$}
						\tNamedPoint{5,4}{below}{$x_{s+1}$}
						\tNamedPoint{0,3}{above}{$x_s$}
						\tNamedPoint{4,1}{above}{$x_{s+2}$}
				}}
			}
			
			\begin{tikzpicture}[scale=0.4]
				\path (0,0) pic[scale=0.4] {common-max-tall};
				\tNamedPoint{3,5}{below}{$y_1$}
				\tNamedPoint{2,2}{above}{$y_2$}
			\end{tikzpicture}
			\caption{Arrangement of $x_{s-1}, x_s, y_1, y_2, x_{s+1}, x_{s+2}$ in \Cref{p:max-tall}. The shaded areas must be empty, since $X$ is tall.}\label{fig:max-tall-help}
		\end{figure}
	
		Let $y_1'$ and $y_2'$ be 1-entries in $P$ such that
		\begin{enumerate}[(a)]
			\item $y_2'$ is to the left of $y_1'$;\label{prop:max-tall:y-h}
			\item $y_1'$ is above or equal to $y_1$ and $y_2'$ is below or equal to $y_2$; and\label{prop:max-tall:y-v}
			\item $\vd(y_1', y_2')$ is maximal under the previous two conditions.\label{prop:max-tall:y-dv-max}
		\end{enumerate}
		
		Let $Y' = (x_1, x_2, \dots, x_s, y_1', y_2', x_{s+1}, \dots, x_m)$. We first show that $Y'$ is a traversal. $Y'$ clearly satisfies (\ref{prop:trav-ltbr}). Since $y_1'$ is not below $y_1$, it is above $x_{s+1}$, so tallness of $X$ implies that $y_1'$ is to the left of $x_{s+2}$. Symmetrically, $y_2'$ is to the right of $x_{s-1}$, so (\ref{prop:max-tall:y-h}) implies $x_{s-1} \lth y_2' \lth y_1' \lth x_{s+2}$, and thus $Y'$ satisfies (\ref{prop:trav-hor}).
		
		Since $y_1'$ is to the right of $x_{s-1}$, tallness of $X$ implies that $y_1'$ is below $x_{s-1}$. We already \linebreak observed that $y_1'$ is above $x_{s+1}$, so we have $x_{s-1} \ltv y_1' \ltv x_{s+1}$. Similarly, we \linebreak have ${x_s \ltv y_s' \ltv x_{s+2}}$. Together with $x_{s+1} \ltv x_s$, this implies the remaining traversal properties (\ref{prop:trav-upper}), (\ref{prop:trav-lower}), (\ref{prop:trav-vert}).
		
		It remains to show that $Y'$ is tall. Suppose $Y$ violates tallness property (\ref{prop:trav-tall-below}). Since $X$ is tall, the only way this can happen is if there is a 1-entry $z$ below $y_2'$ and to the left of $y_1'$. Then $z$ is also below $y_2$, but $\vd(y_1', z) > \vd(y_1', y_2')$, violating our assumption (\ref{prop:max-tall:y-dv-max}). A symmetric argument shows that $Y$ satisfies (\ref{prop:trav-tall-above}).
	\end{proof}

	\subsection{Construction}\label{sec:compl-constr}
	
	Fix a $k \times k$ permutation matrix $P$. Throughout this subsection, we write $\ell, b, t, r$ for $\ell_P, b_P, t_P, r_P$. For a 1-entry $x = (i,j) \in E(P)$, denote by $P^\rL_x$ the submatrix of $P$ consisting of all columns to the left of $x$ (i.e., the leftmost $j-1$ columns), and denote by $P^\rR_x$ the submatrix of $P$ consisting of all columns to the right of $x$ (i.e., the rightmost $k-j$ columns). Note that in both $P^\rL_x$ and~$P^\rR_x$, the $i$-th row is empty. Also note that the constructions in \cref{sec:4-trav,sec:surround} implicitly used $P^\rL_x, P^\rR_x$, with $x \in \{\ell, r, q\}$.
	
	We first construct a matrix with an expandable row. For this, we take $P^\rL_x$ and $P^\rR_x$ for (almost) every 1-entry $x$ in a traversal of $P$ and arrange them like in \cref{fig:cons_ex}~\emph{k}. Then, we vertically move parts of to arrive at our final construction. A formal description follows.
	
	Let $X = (x_1, x_2, \dots, x_m)$ be a traversal of $P$ with $m \ge 6$, and write $(i_s, j_s) = x_s$ \linebreak for~$s \in [m]$. Then the $(2k-1) \times (m-2)k$ matrix $S'(P, X)$ is constructed as follows.
	Let~$L_s'$ be the $(2k-1) \times (j_s-1)$ matrix consisting of a copy of $P^\rL_{x_s}$ that is shifted down by $k-i_s$ rows (i.e., we prepend $k-i_s$ rows and append $i_s-1$ rows to $P^\rL_{x_s}$). Similarly, let $R_s'$ be the $(2k-1) \times (k-j_s)$ matrix consisting of a copy of $P^\rR_{x_s}$ that is shifted down by $k-i_s$ rows.
	Note that the empty $i_s$-th row of $P^\rL_{x_s}$ ($P^\rR_{x_s}$) corresponds to the $k$-th row of $L_s'$ ($R_s'$).
	Finally, let $S'(P,X)$ be the following horizontal concatenation\footnote{See \cpageref{def:concatenation} for the definition of horizontal concatenations.} of matrices:
	\begin{align*}
		S'(P, X) = (L_3', R_1', L_4', R_3', L_5', R_4', \dots, L_{m-3}', R_{m-4}', L_{m-2}', R_{m-3}', L_m', R_{m-2}').
	\end{align*}
	
	\begin{figure}
		\centering
		\begin{tikzpicture}[
				scale=0.2,
				font=\small,
				point/.style={circle, fill, inner sep=1.2pt}
			]
			\begin{scope}[shift={(-1,6)}]
				\tNamedPoint{1,-2}{left}{$x_1$}
				\tNamedPoint{3,-1}{above}{$x_2$}
				\tNamedPoint{2,-4}{below left}{$x_3$}
				\tNamedPoint{5,-3}{above right}{$x_4$}
				\tNamedPoint{4,-6}{below}{$x_5$}
				\tNamedPoint{6,-5}{right}{$x_6$}
			\end{scope}
			\begin{scope}[shift={(18,0)}]
				\draw[dotted] (0,4) -- (24,4);
				
				\node at (-0.5,-2) {$L_3'$};
				\draw (1,-3) -- (1,8);
				\node at (4,-2) {$R_1'$};
				\draw (7,-3) -- (7,8);
				\node at (9.5,-2) {$L_4'$};
				\draw (12,-3) -- (12,8);
				\node at (15,-2) {$R_3'$};
				\draw (17,-3) -- (17,8);
				\node at (20,-2) {$L_6'$};
				\draw (23,-3) -- (23,8);
				\node at (24.5,-2) {$R_4'$};
				
				\node[point] at (24,2) {};
				\node[point] at (22,6) {};
				\node[point] at (21,3) {};
				\node[point] at (20,8) {};
				\node[point] at (19,5) {};
				\node[point] at (18,7) {};
				\node[point] at (16,3) {};
				\node[point] at (15,5) {};
				\node[point] at (14,2) {};
				\node[point] at (13,7) {};
				\node[point] at (11,1) {};
				\node[point] at (10,6) {};
				\node[point] at (9,3) {};
				\node[point] at (8,5) {};
				\node[point] at (6,1) {};
				\node[point] at (5,3) {};
				\node[point] at (4,0) {};
				\node[point] at (3,5) {};
				\node[point] at (2,2) {};
				\node[point] at (0,6) {};
			\end{scope}
		\end{tikzpicture}
		\caption{A matrix $P$ consisting of a 6-traversal $X$, and the corresponding construction~$S'(P,X)$. Some empty columns in $S'(P,X)$ have been omitted. The dotted line indicates the expandable row.}\label{fig:S'_exm}
	\end{figure}
	
	Note the irregularities at the beginning and the end. Notably, $L_2', R_2', L_{m-1}', R_{m-1}'$ are not used in the construction. $L_1'$ and $R_m'$ are not used, either, but they are empty anyway, since $x_1 = \ell$ and $x_m = r$. See \cref{fig:S'_exm} for an example.
	
	We claim that the $k$-th row of $S'(P,X)$ is expandable. Indeed, for each $i$ with $3 \le i \le m-2$, adding a 1-entry in the $k$-th row between $L_i'$ and $R_i'$ will complete a copy of $P$ with $L_i'$ and $R_i'$. Moreover, adding a 1-entry in the $k$-th row to the left of $R_1'$ or to the right of $L_m'$ will complete a copy of $P$. By construction, this covers the whole $k$-th row.
	
	As in the \cref{sec:sur-wit}, we will not directly use $S'(P,X)$, but rather a modified construction that preserves the expandable row.
	We do this to avoid that two different parts of the construction (i.e., $L_i' \cup R_i'$ for $i \in [m]$) overlap vertically. This reduces the number of ways $P$ could appear in the constructed matrix, and thus makes the analysis much easier.
	
	In the following, we will slightly abuse the notation by writing $L_s'$ ($R_s'$) for the subsets of~$E(S'(P,X))$ that correspond to $L_s'$ ($R_s'$).
	
	Let $S(P,X)$ be a $((2m-6)k+1) \times (m-2)k$ matrix, constructed as follows.
	Start with a copy of $S'(P,X)$, shifted down by $(m-4)k$ rows, such that the expandable $k$-th row of~$S'(P,X)$ corresponds to the $(m-3)k$-th row of $S(P,X)$.
	Now, for each ${s \in \{5,6,\dots, m-1, m-2, m\}}$, take all 1-entries in $L_s' \cup R_s'$ that are above the $((m-3)k-1)$-th row (i.e., at least two rows above the expandable row), and move them up by $(s-4)k$ rows.
	Similarly, for each $s \in \{1, 3, 4, \dots, m-4 \}$, take all 1-entries in $L_s' \cup R_s'$ that are below the $((m-3)k+1)$-th row (i.e., at least two rows below the expandable row), and move them down by $(m-s-3)k$ rows. \Cref{fig:big-constr-sketch} shows the rough structure of $S(P,X)$ when $m = 12$ and $X$ is tall.
	
	Let $L_s$ ($R_s$) denote the the modified set of entries in $S(P,X)$ corresponding to $L_s'$ ($R_s'$). Clearly, $L_s$ and $R_s$ still form a partial occurrence of $P$ with a single 1-entry missing between them in the $(m-3)k$-th row. Similarly, $R_1$ and $L_m$ form occurrences when adding a 1-entry in the left- or rightmost part of that row. Thus:
	\begin{lemma}\label{p:s-exp}
		If $X$ is a traversal of $P$, then $S(P,X)$ has an expandable row.
	\end{lemma}
	
	Note that the construction used in \Cref{sec:4-trav} can be seen as a special case of both $S(P,X)$ and $S'(P,X)$ when $m=4$.
	
	The remainder of this paper is dedicated to the proof that if $X$ is a non-extendable tall traversal of a permutation matrix $P$, then $S(P,X)$ avoids $P$, implying that $S(P,X)$ is a vertical witness of $P$.
	We first fix some notations and make a few observations about $S(P,X)$. Let $T$ denote the set of 1-entries that are above row $(m-3)k-1$ (at least two rows above the expandable row). Similarly, let $B$ denote the set of 1-entries that are below row $(m-3)k+1$, and let $M$ denote the remaining 1-entries. For a subset $A \subseteq E(S(P,X))$, let $A^\rT = A \cap T$, let $A^\rB = A \cap B$ and let $A^\rM = A \cap M$. For a 1-entry $p \neq x_s$, let $p^s$ denote the copy of $p$ in $L_s \cup R_s$.
	
	\begin{observation}\label{p:alignment}
		Let $s, u \in \{1,3,4, \dots, m-3, m-2, m\}$ with $s < u$. If $u \ge 5$, then every 1-entry in $L_s^\rT \cup R_s^\rT$ is below every 1-entry in $L_u^\rT \cup R_u^\rT$.
		Moreover, if $s \le m-4$, then every 1-entry in $L_s^\rB \cup R_s^\rB$ is below every 1-entry in $L_u^\rB \cup R_u^\rB$.\qed
	\end{observation}
	
	Since $X$ is tall, there are no 1-entries below and to the left of $x_s$ if $s$ is odd, or above and to the right of $x_s$ if $s$ is even. This implies:
	\begin{observation}\label{p:empty-boxes}
		Let $s$ be odd with $3 \le s \le m-3$. Then $L_s$ contains no 1-entries below the expandable row, and $R_{s+1}$ contains no 1-entries above the expandable row.
		In particular,~${L_s^\rB = \varnothing}$ and $R_{s+1}^\rT = \varnothing$.\qed
	\end{observation}
	
	We now consider the width and height of relevant parts of $S(P,X)$.
	\begin{observation}\label{p:box-sizes}
		For each $s \in \{1, 3,4,\dots,m-3, m-2, m\}$,
		\begin{itemize}
			\item $\width(L_s) = \hd(\ell, x_s)-1$;
			\item $\width(R_s) = \hd(x_s, r)-1$;
			\item $\height(L_s^\rT \cup R_s^\rT) = \vd(t, x_s)-2$, if $L_s^\rT \cup R_s^\rT \neq \varnothing$;
			\item $\height_\phi(L_s^\rM \cup R_s^\rM) \le 1$; and
			\item $\height(L_s^\rB \cup R_s^\rB) = \vd(x_s, b)-2$, if $L_s^\rB \cup R_s^\rB \neq \varnothing$.\qed
		\end{itemize}
	\end{observation}
	
	Let $3 \le s \le m-3$ be odd. Since $X$ is tall, there are no 1-entries in $P$ above $x_{s-1}$ and to the right of $x_s$. Thus, $x_{s-1}^s$ is the topmost 1-entry in $R_s$. Similarly, $x_{s+2}^{s+1}$ is the bottommost 1-entry in $L_{s+1}$. This implies the following improved bounds:
	\begin{observation}\label{p:box-sizes-special}
		For each odd $s \in \{3,4,\dots,m-2\}$:
		\begin{itemize}
			\item $\height(R_s^\rT) \le \vd(x_{s-1}, x_s) - 2$, if $R_s^\rT \neq \varnothing$; and
			\item $\height(L_{s+1}^\rB) \le \vd(x_{s+1}, x_{s+2}) - 2$, if $L_{s+1}^\rB \neq \varnothing$.\qed
		\end{itemize}
	\end{observation}
	
	\begin{figure}
		\centering
		\begin{tikzpicture}[
				scale=0.3,
				midline/.style={dashed},
				label/.style={font=\footnotesize}
			]
			\draw[midline] (2,0) -- (3,0);
\draw (2,0.5) rectangle (3,1.5);
\node[label] at (2.5,2.5) {$R_{1}^\rT$};
\draw[midline] (2,-1) -- (3,-1);
\draw (2,-9.5) rectangle (3,-8.5);
\node[label] at (2.5,-10.5) {$R_{1}^\rB$};
\draw[midline] (0,0) -- (1,0);
\draw (0,0.5) rectangle (1,1.5);
\node[label] at (0.5,2.5) {$L_{3}^\rT$};
\draw[midline] (6,0) -- (7,0);
\draw (6,0.5) rectangle (7,1.5);
\node[label] at (6.5,2.5) {$R_{3}^\rT$};
\draw[midline] (6,-1) -- (7,-1);
\draw (6,-8.5) rectangle (7,-7.5);
\node[label] at (6.5,-9.5) {$R_{3}^\rB$};
\draw[midline] (4,0) -- (5,0);
\draw (4,0.5) rectangle (5,1.5);
\node[label] at (4.5,2.5) {$L_{4}^\rT$};
\draw[midline] (4,-1) -- (5,-1);
\draw (4,-7.5) rectangle (5,-6.5);
\node[label] at (4.5,-8.5) {$L_{4}^\rB$};
\draw[midline] (10,-1) -- (11,-1);
\draw (10,-7.5) rectangle (11,-6.5);
\node[label] at (10.5,-8.5) {$R_{4}^\rB$};
\draw[midline] (8,0) -- (9,0);
\draw (8,1.5) rectangle (9,2.5);
\node[label] at (8.5,3.5) {$L_{5}^\rT$};
\draw[midline] (14,0) -- (15,0);
\draw (14,1.5) rectangle (15,2.5);
\node[label] at (14.5,3.5) {$R_{5}^\rT$};
\draw[midline] (14,-1) -- (15,-1);
\draw (14,-6.5) rectangle (15,-5.5);
\node[label] at (14.5,-7.5) {$R_{5}^\rB$};
\draw[midline] (12,0) -- (13,0);
\draw (12,2.5) rectangle (13,3.5);
\node[label] at (12.5,4.5) {$L_{6}^\rT$};
\draw[midline] (12,-1) -- (13,-1);
\draw (12,-5.5) rectangle (13,-4.5);
\node[label] at (12.5,-6.5) {$L_{6}^\rB$};
\draw[midline] (18,-1) -- (19,-1);
\draw (18,-5.5) rectangle (19,-4.5);
\node[label] at (18.5,-6.5) {$R_{6}^\rB$};
\draw[midline] (16,0) -- (17,0);
\draw (16,3.5) rectangle (17,4.5);
\node[label] at (16.5,5.5) {$L_{7}^\rT$};
\draw[midline] (22,0) -- (23,0);
\draw (22,3.5) rectangle (23,4.5);
\node[label] at (22.5,5.5) {$R_{7}^\rT$};
\draw[midline] (22,-1) -- (23,-1);
\draw (22,-4.5) rectangle (23,-3.5);
\node[label] at (22.5,-5.5) {$R_{7}^\rB$};
\draw[midline] (20,0) -- (21,0);
\draw (20,4.5) rectangle (21,5.5);
\node[label] at (20.5,6.5) {$L_{8}^\rT$};
\draw[midline] (20,-1) -- (21,-1);
\draw (20,-3.5) rectangle (21,-2.5);
\node[label] at (20.5,-4.5) {$L_{8}^\rB$};
\draw[midline] (26,-1) -- (27,-1);
\draw (26,-3.5) rectangle (27,-2.5);
\node[label] at (26.5,-4.5) {$R_{8}^\rB$};
\draw[midline] (24,0) -- (25,0);
\draw (24,5.5) rectangle (25,6.5);
\node[label] at (24.5,7.5) {$L_{9}^\rT$};
\draw[midline] (30,0) -- (31,0);
\draw (30,5.5) rectangle (31,6.5);
\node[label] at (30.5,7.5) {$R_{9}^\rT$};
\draw[midline] (30,-1) -- (31,-1);
\draw (30,-2.5) rectangle (31,-1.5);
\node[label] at (30.5,-3.5) {$R_{9}^\rB$};
\draw[midline] (28,0) -- (29,0);
\draw (28,6.5) rectangle (29,7.5);
\node[label] at (28.5,8.5) {$L_{10}^\rT$};
\draw[midline] (28,-1) -- (29,-1);
\draw (28,-2.5) rectangle (29,-1.5);
\node[label] at (28.5,-3.5) {$L_{10}^\rB$};
\draw[midline] (34,-1) -- (35,-1);
\draw (34,-2.5) rectangle (35,-1.5);
\node[label] at (34.5,-3.5) {$R_{10}^\rB$};
\draw[midline] (32,0) -- (33,0);
\draw (32,7.5) rectangle (33,8.5);
\node[label] at (32.5,9.5) {$L_{12}^\rT$};
\draw[midline] (32,-1) -- (33,-1);
\draw (32,-2.5) rectangle (33,-1.5);
\node[label] at (32.5,-3.5) {$L_{12}^\rB$};
			\node[label,left] at (0,-0.5) {$M$};
			\node[label,right] at (35,-0.5) {$M$};
		\end{tikzpicture}
		\caption{A sketch of the block structure of $S(P,X)$ with $|X| = 12$.}\label{fig:big-constr-sketch}
	\end{figure}
	
	\subsection{\texorpdfstring{$S(P,X)$}{S(P,X)} avoids \texorpdfstring{$P$}{P}}
	
	In this section, we show:
	\begin{lemma}\label{p:big-S-avoids}
		Let $P$ be a permutation matrix, $m \ge 6$ be even and let $X = (x_1, x_2, \dots, x_m)$ be a non-extendable tall traversal of $P$. Then $S(P,X)$ avoids $P$.
	\end{lemma}
	
	Together with \Cref{p:osc-is-trav,p:max-tall,p:s-exp}, this implies \Cref{p:even-6-l-osc}. For the remainder of this section, fix $P$ and $X$ as in \cref{p:big-S-avoids}, and write $\ell, b, t, r$ for $\ell_P, b_P, t_P, r_P$. We use the notation for parts of $S(P,X)$ as defined in \Cref{sec:compl-constr}. Suppose $\phi$ is an embedding of $P$ into $S(P,X)$. Our overall strategy is to distinguish cases based on the location of $\phi(t)$, and derive a contradiction in each case. While the full proof is long and technical, it only uses a handful of simple arguments that are combined and applied to various situations.
	
	Note that we make no further assumptions on $P,X,\phi$, so each lemma or corollary in this section holds on its own for every choice of $P,X,\phi$ (we only fix $P, X, \phi$ for brevity). This allows us to make use of the following symmetry argument. $S(P,X)$ is not usually symmetric, in the sense that its 180-degree rotation $\rot^2(S(P,X))$ is equal to $S(P,X)$. However, it is easy to see that $\rot^2(S(P,X))$ is equal to $S(\rot^2(P),\rot^2(X))$. Now, in \cref{p:t-not-L_3}, for example, we show that $\phi(t) \notin L_3$ for each choice of $P, X, \phi$, in particular also for every embedding $\phi'$ of~$\rot^2(P)$ into $\rot^2(S(P,X))$.
	
	We also get $\phi(b) \notin R_{m-2}$, since $R_{m-2}$ in $S(P,X)$ corresponds to $L_3$ in \linebreak $\rot^2(S(P,X)) = S(\rot^2(P),\rot^2(X))$, $b$~in $P$ corresponds to $t$~in $\rot^2(P)$, and $\phi$ corresponds to some embedding $\phi'$ of $\rot^2(P)$ into $\rot^2(S(P,X))$.
	
	\subsubsection{\texorpdfstring{$\phi(t)$}{phi(t)} in the front or the back}
	
	In this section, we show $\phi(t)$ and $\phi(b)$ cannot lie in the leftmost or rightmost few ``blocks'' of~$S(P,X)$. The precise results are \cref{p:t-not-borders,p:t-not-R3-or-b} at the end of the section.
	The proofs in this section also serve as a warm-up for the more complex later proofs. Most techniques used in \cref{sec:middle_diffgroup,sec:middle_samegroup} already appear here, where we explain them thoroughly.
	
	\begin{lemma}\label{p:t-not-L_3}
		$\phi(t) \notin L_3$ and $\phi(b) \notin R_{m-2}$.
	\end{lemma}
	\begin{proof}
		By symmetry, it suffices to show $\phi(t) \notin L_3$. Suppose $\phi(t) \in L_3$. Then \linebreak also $\phi(\ell) \in L_3$, since $S(P,X)$ contains no 1-entries to the left of $L_3$. But \linebreak ${\width(L_3) = \hd(\ell, x_3) -1 < \hd(\ell,t)-1}$, thus we cannot have both $\phi(\ell)$ and $\phi(t)$ in $L_3$, a contradiction.
	\end{proof}
	
	\begin{lemma}\label{p:t-not-R_1}
		$\phi(t) \notin R_1$ and $\phi(b) \notin L_m$.
	\end{lemma}
	\begin{proof}
		By symmetry, it suffices to show $\phi(t) \notin R_1$. Suppose $\phi(t) \in R_1$.
		Note \linebreak that $\height(R_1^\rT \cup M) \le \vd(t, \ell)+1$. Since this bound counts the (empty) expandable row, we have $\height_\phi(R_1^\rT \cup M) \le \vd(t, \ell)$. This implies that $\phi(\ell)$ must lie in the lowest row of $M$ or lower, and thus $\phi(\ell)$ is below the expandable row.
		
		Since $x_3$ is below $\ell$, this also implies that $x_3$ at least two rows below the expandable row, so~$\phi(x_3) \in B$. Further, $x_3$ is to the right of $t$ and $L_3^\rB = \varnothing$, so we have $\phi(x_3) \in R_1^\rB$. As~$r$ is below~$x_3$, and all 1-entries in $S(P,X)$ that are to the right of $R_1$ are above $R_1^\rB$, we \linebreak have~$\phi(r) \in R_1^\rB$. Since $\width(R_1) < \hd(\ell,r)$, this implies that $\phi(\ell)$ is to the left of $R_1$. We now know that $\phi(\ell)$ is below the expandable row and to the left of $R_1^\rT$. But $S(P,X)$ has no such 1-entry, a contradiction.
	\end{proof}
	
	If $m=6$ (see \cref{fig:small-constr-sketch}), then the only remaining possibility is $\phi(t), \phi(b) \in L_4 \cup R_3$, which implies $\phi(t) \in L_4$ or $\phi(b) \in R_3$ (since $t$ is to the left of $b$). Thus, the following lemma concludes the case $m=6$.
	
	\begin{figure}
		\centering
		\begin{tikzpicture}[
			scale=0.3,
			midline/.style={dashed},
			label/.style={font=\footnotesize}
			]
			\draw[midline] (2,0) -- (3,0);
\draw (2,0.5) rectangle (3,1.5);
\node[label] at (2.5,2.5) {$R_{1}^\rT$};
\draw[midline] (2,-1) -- (3,-1);
\draw (2,-3.5) rectangle (3,-2.5);
\node[label] at (2.5,-4.5) {$R_{1}^\rB$};
\draw[midline] (0,0) -- (1,0);
\draw (0,0.5) rectangle (1,1.5);
\node[label] at (0.5,2.5) {$L_{3}^\rT$};
\draw[midline] (6,0) -- (7,0);
\draw (6,0.5) rectangle (7,1.5);
\node[label] at (6.5,2.5) {$R_{3}^\rT$};
\draw[midline] (6,-1) -- (7,-1);
\draw (6,-2.5) rectangle (7,-1.5);
\node[label] at (6.5,-3.5) {$R_{3}^\rB$};
\draw[midline] (4,0) -- (5,0);
\draw (4,0.5) rectangle (5,1.5);
\node[label] at (4.5,2.5) {$L_{4}^\rT$};
\draw[midline] (4,-1) -- (5,-1);
\draw (4,-2.5) rectangle (5,-1.5);
\node[label] at (4.5,-3.5) {$L_{4}^\rB$};
\draw[midline] (10,-1) -- (11,-1);
\draw (10,-2.5) rectangle (11,-1.5);
\node[label] at (10.5,-3.5) {$R_{4}^\rB$};
\draw[midline] (8,0) -- (9,0);
\draw (8,1.5) rectangle (9,2.5);
\node[label] at (8.5,3.5) {$L_{6}^\rT$};
\draw[midline] (8,-1) -- (9,-1);
\draw (8,-2.5) rectangle (9,-1.5);
\node[label] at (8.5,-3.5) {$L_{6}^\rB$};
			\node[label,left] at (0,-0.5) {$M$};
			\node[label,right] at (11,-0.5) {$M$};
		\end{tikzpicture}
		\caption{A sketch of the block structure of $S(P,X)$ with $|X| = 6$.}\label{fig:small-constr-sketch}
	\end{figure}

	\begin{lemma}\label{p:m6-t-not-L_4}
		If $m = 6$, then $\phi(t) \notin L_4$ and $\phi(b) \notin R_3$.
	\end{lemma}
	\begin{proof}
		By symmetry, it suffices to show $\phi(t) \notin L_4$. This can be done with essentially the same argument as in the proof of \cref{p:4_trav_wit}. Suppose $\phi(t) \in L_4$. Then $\phi(t)$ is not above $t^4 \in L_4$. By \cref{p:t-not-L_3,p:t-not-R_1}, $\phi(b) \in L_4 \cup R_3$. The lowest 1-entry in $L_4 \cup R_3$ is $b^4$, so $\phi(b)$ is not below $b^4$. But $\vd_\phi(t^4, b^4) < \vd(t,b)$ (note the empty expandable row), a contradiction.
	\end{proof}
	
	We now continue with the case $m \ge 8$.
	
	\begin{lemma}\label{p:t-not-L_4}
		If $m \ge 8$, then $\phi(t) \notin L_4$ and $\phi(b) \notin R_{m-3}$.
	\end{lemma}
	\begin{proof}
		By symmetry, it suffices to show $\phi(t) \notin L_4$. Suppose $\phi(t) \in L_4$. We have \linebreak $\height_\phi(L_4^\rT \cup M) \le \vd(t,x_4) < \vd(t, x_3)$, implying that $\phi(x_3) \in B$. More precisely, we have $\phi(x_3) \in R_1^\rB \cup L_4^\rB$, because $x_3$ is to the left of $t$.
		
		Since $r$ is below $x_3$ and to the left of $t$, we have $\phi(r) \in L_4^\rB \cup R_3^\rB \cup R_4^\rB$. This means that~$\phi$ maps no 1-entry to the right of $R_4$, and thus maps no 1-entry into the rows between $M$ and~$L_4^\rB \cup R_4^\rB$. This is a very useful observation, since it essentially allows us to pretend that $M$ is directly above~$L_4^\rB \cup R_4^\rB$. Similar observations will be used frequently in subsequent proofs.
		
		From the above, we get $\height_\phi(L_4 \cup R_4) < \vd(t, b)$ (note that $L_4^\rT$ is directly above $M$), so~$\phi(b)$ is below $L_4 \cup R_4$, and thus $\phi(b) \in R_3^\rB$. Moreover, by tallness of $X$, we have \linebreak $\height_\phi(L_4) < \vd(t, x_5)$, so $\phi(x_5)$ is below $L_4$. Since $x_5$ is to the left of $b$, this means \linebreak that~$\phi(x_5) \in R_3^\rB$.
		
		Consider now $\phi(x_4)$. Since $x_5 \lth x_4 \lth b$, we have $\phi(x_4) \in R_3$. Since \linebreak $\height(R_3^\rB) < \vd(x_3,b) < \vd(x_4,b)$, we have $\phi(x_4) \in R_3^\rT \cup R_3^\rM$.
		
		We conclude the proof with a case distinction. First, assume that $\phi(t) \neq t^4$. Since $\phi(t) \in L_4$ and $t^4$ is the highest 1-entry in $L_4$, this means that $\phi(t)$ is below $t^4$, and thus \linebreak $\vd_\phi(\phi(t), \phi(x_4)) < \vd_\phi(t^4, \phi(x_4)) \le \height_\phi(L_4 \cup M) \le \vd(t, x_4)$, a contradiction.
		
		Second, assume that $\phi(t) = t_4$. Recall that $\phi$ maps no 1-entries between $M$ and $L_4^\rB$. Because of this and the fact that the expandable row is empty, we have that $\vd_\phi(t^4, x_3^4) < \vd(t, x_3)$, implying that $\phi(x_3)$ is below $x_3^4$. By tallness of $X$, this also implies that $\phi(x_3)$ is to the right of $x_3^4$. However, since $x_3^4$ is to the left of $t^4$, this means that $\hd(\phi(x_3), \phi(t)) < \hd(x_3^4, t^4) = \hd(x_3, t)$, a contradiction.
	\end{proof}
	
	\begin{lemma}\label{p:t-not-R_3_except}
		Let $m \ge 8$. If $\phi(t) \in R_3$, then $\phi(b)$ is to the right of $R_4$. Moreover, if $\phi(b) \in L_{m-2}$, then $\phi(t)$ is to the left of $L_{m-3}$.
	\end{lemma}
	\begin{proof}
		By symmetry, proving the first statement suffices. Let $\phi(t) \in R_3$ and suppose $\phi(b)$ not to the right of $R_4$. Since $\phi(R_3^\rT) \le \vd(t, x_3)$, we know that $\phi(x_3)$ is below the expandable row. Let $q_3$ be the 1-entry directly below $x_3$ in $P$. Clearly, $\phi(q_3), \phi(b), \phi(r) \in B$, and since $\phi(b)$ is to the right of $\phi(t)$ and not to the right of $R_4$, we have $\phi(b) \in R_3^\rB \cup R_4^\rB$. We separately consider three cases.
		\begin{casedist}
			\item $\phi(r) \in R_3^\rB$. Since $X$ is tall, $q_3$ is to the right of $t$, so $\phi(q_3) \in R_3^\rB$. But \linebreak $\height(R_3^\rB) = \vd(x_3,b) - 2 = \vd(q_3,b)-1$, a contradiction.
			\item $\phi(r) \in R_4^\rB$. Consider $x_5$. Since $x_5$ is below $x_3$, we have $\phi(x_5) \in B$. Since $x_5$ is \linebreak to the right of $t$, and above and to the left of $r$, we have $\phi(x_5) \in R_4^\rB$. But \linebreak $\width(R_4) = \hd(x_4, r)-1 < \hd(x_5,r)$, a contradiction.
			\item $\phi(r)$ is to the right of $R_4$. Then $\phi(r)$ is also above $L_4^\rB \cup R_4^\rB$. Consider again $x_5$. We know that $\phi(x_5)$ is below $M$ and above $L_4^\rB \cup R_4^\rB$. Since $x_5$ is to the left of $b$, we also know that $\phi(x_5)$ is not to the right of $R_4$. But there are no such 1-entries in $S(P,X)$, a contradiction.\qedhere
		\end{casedist}
	\end{proof}
	
	We proceed with some more special cases, showing that $\phi(t)$ also cannot lie in the rightmost few blocks of $S(P,X)$.
	
	\begin{lemma}\label{p:t-left-of-L_m-2}
		Let $m \ge 8$. Then, $\phi(t)$ lies to the left of $L_{m-2}$, and $\phi(b)$ lies to the right of $R_3$.
	\end{lemma}
	\begin{proof}
		By symmetry, it suffices to prove that $\phi(t)$ lies to the left of $L_{m-2}$. If $\phi(b)$ lies to the left of $L_{m-2}$, then $\phi(t)$ does, too. $\phi(b) \notin R_{m-3} \cup L_m \cup R_{m-2}$ by \Cref{p:t-not-L_3,p:t-not-R_1,p:t-not-L_4}. The only remaining possibility is that $\phi(b) \in L_{m-2}$, where \Cref{p:t-not-R_3_except} implies that $\phi(t)$ lies to the left of $L_{m-3}$, and thus to the left of $L_{m-2}$.
	\end{proof}
	
	To show that $\phi(t) \notin L_{m-3} \cup R_{m-4}$, we use the following more general lemma. \Cref{fig:big-constr-middle-sketch} is useful to visualize the proof.
	\begin{lemma}\label{p:not-t-L_s-b-R_s-1}
		Let $s$ be odd with $5 \le s \le m-3$. If $\phi(t) \in L_s \cup R_{s-1}$, then $\phi(b)$ lies to the right of $R_{s-1}$.
	\end{lemma}
	\begin{proof}
		Suppose not. Then, $\phi(t), \phi(b) \in L_{s} \cup R_{s-1}$.
		
		\begin{casedist}
			\item $\phi(\ell) \notin L_s \cup R_{s-1}$. Since $\ell$ is to the left of $t$, this means that $\phi(\ell)$ is to the left of $L_{s}$. \linebreak This implies that $\phi(\ell)$ is also below $L_s^\rT$, and thus $\phi(x_4)$ is below $L_s^\rT$. \linebreak Since $x_4$ is to the right of $t$, we have $\phi(x_4) \in M \cup B$, which implies $\phi(x_5) \in B$, as \linebreak $\height_\phi(M) \le 1 < \vd(x_4, x_5)$. Since $x_5$ is to the right of $t$ and to the left of $b$, we further know $\phi(x_5) \in R_{s-1}^\rB$. Since $\width(R_{s-1}) < \hd(x_5, r)$, this implies that $\phi(r)$ is to the right of $R_{s-1}$. But then $\phi(r)$ is above $\phi(x_5)$, a contradiction.
			
			\item $\phi(\ell) \in L_s \cup R_{s-1}$. Then $\phi$ maps no 1-entry to the left of $L_s$.
			Since $P$ is indecomposable, there must be some $y, z \in E(P)$ such that $\phi(y) \in L_s$, and $\phi(z)$ is above~$\phi(y)$ and to the right of $L_s$. Note that $L_s$ contains no 1-entries below the expandable row (by tallness of $X$), so $\phi(z) \in T$. Further, $\phi(t) \in L_s$ implies that~$\phi(z) \in R_s^\rT$.
			Since~$\phi(b) \in L_s \cup R_{s-1}$, we know that $b$ is to the left of $z$. Tallness of~$X$ implies that $z$ is not above $x_{m-2}$. Now consider $x_{s-1}$. We know $x_{s-1} \leh x_{m-4} \lth b$ \linebreak and~$x_{s-1} \lev x_{m-4} \ltv x_{m-2} \lev z$. Thus, $\phi(x_{s-1}) \in L_s^\rT$. But \linebreak $\width(L_s^\rT) < \vd(\ell, x_s) < \vd(\ell, x_{s-1})$, a contradiction.\qedhere
		\end{casedist}
	\end{proof}
	
	\begin{corollary}\label{p:t-not-L_m-3}
		If $m \ge 8$, then $\phi(t) \notin L_{m-3} \cup R_{m-4}$ and $\phi(b) \notin L_5 \cup R_4$.
	\end{corollary}
	\begin{proof}
		By symmetry, it suffices to prove that $\phi(t) \notin L_{m-3} \cup R_{m-4}$. Suppose \linebreak ${\phi(t) \in L_{m-3} \cup R_{m-4}}$. By \Cref{p:t-not-L_3,p:t-not-R_1,p:t-not-L_4,p:t-not-R_3_except}, $\phi(b)$ cannot lie in $L_{m-2}$ or further right. This contradicts \Cref{p:not-t-L_s-b-R_s-1}.
	\end{proof}
	
	We now consolidate and reformulate the above results. For the more involved proofs in \linebreak \cref{sec:middle_diffgroup,sec:middle_samegroup}, it will be convenient to organize the ``middle'' blocks $L_i, R_i$ \linebreak of $S(P,X)$ into two sets of groups, as follows. For each odd $s$ with $5 \le s \le m-5$, \linebreak let $G_s = L_s \cup R_{s-1} \cup L_{s+1} \cup R_s$, and let $H_s = L_{s+1} \cup R_s \cup L_{s+2} \cup R_{s+1}$. \cref{fig:big-constr-middle-sketch} illustrates~$G_s$ and $H_s$. Combining \cref{p:t-not-L_3,p:t-not-R_1,p:t-not-L_4,p:t-left-of-L_m-2,p:t-not-L_m-3} yields:
	\begin{corollary}\label{p:t-not-borders}
		If $m \ge 8$, then:
		\begin{itemize}
			\item $\phi(t)$ lies to the right of $L_4$ and to the left of $L_{m-3}$. In other words, $\phi(t) \in R_3$ or $\phi(t) \in G_s$ for some odd $s$ with $5 \le s \le m-5$; and
			\item $\phi(b)$ lies to the right of $R_4$ and to the left of $R_{m-3}$. In other words, $\phi(b) \in L_{m-2}$ \linebreak or~$\phi(b) \in H_s$ for some odd $s$ with $5 \le s \le m-5$.
		\end{itemize}
	\end{corollary}
	
	At this stage, we cannot easily show that both $\phi(t) \notin R_3$ and $\phi(b) \notin L_{m-2}$, but we can show that at least one of the two must be true.
	\begin{lemma}\label{p:t-not-R3-or-b}
		If $m \ge 8$, then $\phi(t) \notin R_3$ or $\phi(b) \notin L_{m-2}$
	\end{lemma}
	\begin{proof}
		Suppose $\phi(t) \in R_3$ and $\phi(b) \in L_{m-2}$. Since $\height_\phi(R_3^\rT \cup M) \le \vd(t, x_3) < \vd(t,x_5)$, we have $\phi(x_5) \! \in \! B$. More precisely, as $b \! \in \! L_{m-2}$, we have ${\phi(x_5) \! \in \! L_{m-2}^\rB \cup R_{m-3}^\rB \cup L_m^\rB \cup R_{m-2}^\rB}$. Similarly, $\phi(x_{m-4}) \in L_3^\rT \cup R_1^\rT \cup L_4^\rT \cup R_3^\rT$. In particular, $\phi(x_5)$ is to the right of $\phi(x_{m-4})$. But~$x_5 \lth x_4 \leh x_{m-4}$, a contradiction.
	\end{proof}
	
	Note that \Cref{p:t-not-borders,p:t-not-R3-or-b} completely resolve the case $m = 8$.
	
	In the following two subsections, we show that the remaining possibilities also lead to a contradiction. In \cref{sec:middle_diffgroup}, we treat the easier case, where $\phi(t) \in G_s$ for some odd $s$ with~$5 \le s \le m-5$, and $\phi(b)$ is to the right of $R_{s+1}$ (i.e., to the right of $H_s$). This also handles the symmetric case where $\phi(b) \in H_s$ and $\phi(t)$ is to the left of $G_s$. In \cref{sec:middle_samegroup}, we consider the case where $\phi(t) \in G_s$ and $\phi(b) \in H_s$.
	
	\begin{figure}
		\centering
		\begin{tikzpicture}[
			scale=0.3,
			midline/.style={densely dashed},
			rest/.style={densely dashed},
			label/.style={font=\footnotesize}
			]
			% Manually edited!

\begin{scope}[shift={(-7,0)}]
	\begin{scope}[shift={(0,1)}]
		\draw (8,1.5) rectangle (9,2.5);
		\node[label] at (8.5,3.5) {$L_s^\rT$};
		\draw (12,2.5) rectangle (13,3.5);
		\node[label] at (12.5,4.5) {$L_{s+1}^\rT$};
		\draw (14,1.5) rectangle (15,2.5);
		\node[label] at (14.5,3.5) {$R_s^\rT$};
		\draw (16,3.5) rectangle (17,4.5);
		\node[label] at (16.5,5.5) {$L_{s+2}^\rT$};
		
		\draw[|-|] (8,6) -- (15,6);
		\node[label, above] at (11.5, 6) {$G_s$};
	\end{scope}
	\begin{scope}[shift={(0,1)}]
		\draw (10,-7.5) rectangle (11,-6.5);
		\node[label] at (10.5,-8.5) {$R_{s-1}^\rB$};
		\draw (12,-5.5) rectangle (13,-4.5);
		\node[label] at (12.5,-6.5) {$L_{s+1}^\rB$};
		\draw (14,-6.5) rectangle (15,-5.5);
		\node[label] at (14.5,-7.5) {$R_s^\rB$};
		\draw (18,-5.5) rectangle (19,-4.5);
		\node[label] at (18.5,-6.5) {$R_{s+1}^\rB$};
		
		\draw[|-|] (12,-9) -- (19,-9);
		\node[label, below] at (15.5, -9) {$H_s$};
	\end{scope}
	\draw[midline] (8,0) -- (9,0);
	\draw[midline] (12,0) -- (13,0);
	\draw[midline] (14,0) -- (15,0);
	\draw[midline] (16,0) -- (17,0);
	\draw[midline] (10,-1) -- (11,-1);
	\draw[midline] (12,-1) -- (13,-1);
	\draw[midline] (14,-1) -- (15,-1);
	\draw[midline] (18,-1) -- (19,-1);
\end{scope}

\draw[rest] (-2,2) -- (0,2) -- (0,0) -- (-2,0);
\draw[rest] (13,7) -- (13,5) -- (15,5);
\draw[rest] (15,-3) -- (13,-3) -- (13,-1) -- (15,-1);
\draw[rest] (-2,-7) -- (0,-7) -- (0,-9);
			
			\begin{scope}[shift={(22, -1)}, font=\footnotesize]
				\tNamedPoint{-1,4}{left}{$\ell$}
				\tNamedPoint{0,2}{below left}{$x_3$}
				\tNamedPoint{1,5}{above right}{$t$}
				\tNamedPoint{2,0}{below left}{$x_5$}
				\tNamedPoint{3,3}{above right}{$x_4$}
				\draw (-1,5) -- (-1,2) -- (1,2) -- (1,0) -- (3,0);
				\draw (-1,5) -- (0,5) -- (2,5) -- (2,3) -- (4,3);
			\end{scope}
			\begin{scope}[shift={(30, -3)}, font=\footnotesize]
				\tNamedPoint{0,2}{below left}{$x_s$}
				\tNamedPoint{1,5}{above right}{$x_{s-1}$}
				\tNamedPoint{2,0}{below left}{$x_{s+2}$}
				\tNamedPoint{3,3}{above right}{$x_{s+1}$}
				\draw (-1,2) -- (1,2) -- (1,0) -- (3,0);
				\draw (0,5) -- (2,5) -- (2,3) -- (4,3);
			\end{scope}
			\begin{scope}[shift={(38, -5)}, font=\footnotesize]
				\tNamedPoint{0,2}{below left}{$x_{m-3}$}
				\tNamedPoint{1,5}{above right}{$x_{m-4}$}
				\tNamedPoint{2,0}{below left}{$b$}
				\tNamedPoint{3,3}{above right}{$x_{m-2}$}
				\tNamedPoint{4,1}{right}{$r$}
				\draw (-1,2) -- (1,2) -- (1,0) -- (4,0);
				\draw (0,5) -- (2,5) -- (2,3) -- (4,3) -- (4,0);
			\end{scope}
		\end{tikzpicture}
		\caption{\emph{(left)} A sketch of parts of $S(P,X)$. Here, $s$ is odd and $5 \le s \le m-5$. The dashed lines and open rectangles indicate $M$ and the rest of $S(P,X)$. \emph{(right)} Sketches of three (not necessarily disjoint) parts of $P$. The solid lines illustrate tallness.}\label{fig:big-constr-middle-sketch}
	\end{figure}

	\subsubsection{\texorpdfstring{$\phi(t)$}{phi(t)}, \texorpdfstring{$\phi(b)$}{phi(b)} in the middle and far from each other}\label{sec:middle_diffgroup}
	
	The following lemma is central to this subsection and will also be useful later on.
	
	\begin{lemma}\label{p:x_s_above_implies}
		Let $s$ be odd with $5 \le s \le m-5$ such that $\phi(t) \in G_s$. Then $\phi(x_s)$ is below the expandable row, or $\phi(\ell), \phi(t), \phi(x_s) \in L_{s+1}$.
	\end{lemma}
	\begin{proof}
		Assume that $\phi(\ell), \phi(t), \phi(x_s) \in L_{s+1}$ does not hold. We show that then $\phi(x_s)$ is below the expandable row. Note that $\phi(t) \in G_s$ implies that $\phi$ maps no 1-entry into $G_u^\rT$ for $u > s$.
		\begin{casedist}
			\item $\phi(\ell) \notin L_s \cup L_{s+1} \cup R_s$. Then $\phi(\ell)$ is below $L_s^\rT \cup R_s^\rT$. Since $x_4$ is to the right of $t$ and below $\ell$, this implies that $x_4 \in M \cup B$. Since $x_4$ is above $x_s$, this means that $\phi(x_s)$ is in the bottom row of $M$ or further below, so $\phi(x_s)$ is below the expandable row.\label{case:s_u_x_s_below_exp:easy}
			
			\item $\phi(t) \notin L_s \cup L_{s+1} \cup R_s$. Then $\phi(t) \in R_{s-1}$, so $\phi(t)$ is below the expandable row, implying the same for $\phi(x_s)$.
			
			\item $\phi(\ell), \phi(t) \in L_s \cup R_s$. Since $\phi$ does not map any 1-entry to a position below $L_s^\rT \cup R_s^\rT$ and above $M$, we have $\height_\phi(L_s^\rT \cup L_s^\rM \cup R_s^\rT \cup R_s^\rM) \le \vd(t, x_s)$. Thus, $\phi(x_s)$ is in the bottom row of $M$ or further below.
			
			\item $\phi(\ell) \in L_s$ and $\phi(t) \in L_{s+1}$. Since $x_4$ is below $\ell$ and to the right of $t$, we have either~$\phi(x_4) \in M \cup B$ or $\phi(x_4) \in R_s^\rT$. In the former case, we are done, as in case~\ref{case:s_u_x_s_below_exp:easy}. In the latter case, note that $\phi$ does not map any 1-entry of $P$ into a row below $L_S^\rT \cup R_s^\rT$ and above $M$, thus $\height_\phi(R_s^\rT \cup R_s^\rM) = \vd(x_{s-1}, x_s) \le \vd(x_4, x_s)$ by \cref{p:box-sizes-special}. Thus, $\phi(x_s)$ is below the expandable row.
			
			\item $\phi(\ell), \phi(t) \in L_{s+1}$ and $\phi(x_s) \notin L_{s+1}$. Since $x_s$ is to the right of $t$, this means that $\phi(x_s)$ is to the right of $L_{s+1}$. Suppose $x_s$ is above the expandable row. Then all 1-entries in~$L_{s+1}^\rM \cup L_{s+1}^\rB$ are below or in the same row as $x_s$. Tallness of $X$, together with the fact that $\phi$ maps no two 1-entries into the same row, implies that $\phi$ maps no 1-entries into $L_{s+1}^\rM \cup L_{s+1}^\rB$.
			But then $\phi$ maps every 1-entry of $P$ either to $L_{s+1}^\rT$ or to the right and below $L_{s+1}^\rT$, and $\phi(t) \in L_{s+1}^\rT, \phi(b) \notin L_{s+1}^\rT$. This means that $P$ is decomposable, a contradiction.\qedhere
		\end{casedist}
	\end{proof}
	
	\begin{lemma}\label{p:s_u_x_s_below_exp}
		For each odd $s$ with $5 \le s \le m-5$, if $\phi(t) \in G_s$ and $\phi(b)$ is to the right of $H_s$, then $\phi(x_s)$ is below the expandable row.
	\end{lemma}
	\begin{proof}
		Suppose $\phi(x_s)$ is above the expandable row. By \cref{p:x_s_above_implies}, ${\phi(\ell), \phi(t), \phi(x_s) \in L_{s+1}}$. Since $\height(L_{s+1}^\rT) < \vd(t, x_s)$, we know that $\phi(x_s)$ is below $L_{s+1}^\rT$, so $\phi(x_s)$ must be in the top row of $L_{s+1}^\rM$.
		
		$x_{s+1}$ is above and to the right of $x_s$, implying that $\phi(x_{s+1}) \in L_{s+1}^\rT \cup R_s^\rT$. Further, $x_{s+2}$ is below $x_s$, so $\phi(x_{s+2}) \in M \cup B$, and $x_s \lth x_{s+2} \lth x_{s+1}$, so $\phi(x_{s+2}) \in L_{s+1} \cup R_s$. Since $\phi(b)$ is to the right of $H_s$ by assumption, we know that $\phi$ maps no 1-entry to $H_s^\rB$. \linebreak Thus,~$\phi(x_{s+2}) \in L_{s+1}^\rM \cup R_s^\rM$.
		
		But now $\phi(x_s), \phi(x_{s+2}) \in L_{s+1}^\rM \cup R_s^\rM$, so $\phi$ maps no further 1-entries to $M$. Therefore, $\phi$ maps every 1-entry either to $A = L_{s+1}^\rT \cup L_{s+1}^\rM \cup R_s^\rT \cup R_s^\rM$, or below and to the right of $A$ (and~$\phi(t) \in A$, $\phi(b) \notin A$). This means $P$ is decomposable, a contradiction.
	\end{proof}
	
	We now consider a simple special case.
	
	\begin{lemma}\label{p:t_Gs_b_right}
		If $\phi(t) \in G_s$ for some odd $s$ with $5 \le s \le m-5$, then $\phi(b) \notin L_{m-2}$.
		
		Moreover, if $\phi(b) \in H_s$ for some odd $s$ with $5 \le s \le m-5$, then $\phi(t) \notin R_3$.
	\end{lemma}
	\begin{proof}
		By symmetry, it suffices to show the first statement. Suppose $\phi(b) \in L_{m-2}$. Then, $\phi(b)$ is to the right of $R_{m-4}$, and thus to the right of $R_{s+1}$, so \cref{p:s_u_x_s_below_exp} implies that $\phi(x_s)$ is below the expandable row. Since $x_s$ is to the left and above $b$, we have $\phi(x_s) \in L_{m-2}$. \linebreak Since $x_s \lth x_{s+1} \lth b$, and $x_{s+1}$ is below $t$, we have $\phi(x_{s+1}) \in L_{m-2}^\rM \cup L_{m-2}^\rB$. But \linebreak $\height_\phi(L_{m-2}^\rM \cup L_{m-2}^\rB) \le \vd(x_{m-2}, b) < \vd(x_{m-4}, b) \le \vd(x_{s+1}, b)$, a contradiction.
	\end{proof}
	
	We proceed with the main case of this subsection.
	
	\begin{lemma}\label{p:t_Gs_b_Hu}
		Let $s, u$ be odd such that $5 \le s < u \le m-5$. If $\phi(t) \in G_s$, then $\phi(b) \notin H_u$.
	\end{lemma}
	\begin{proof}
		Suppose $\phi(b) \in H_u$. Note that then $\phi$ maps no 1-entry to $G_s^\rB$ or $H_u^\rT$.
		
		\Cref{p:s_u_x_s_below_exp} implies that $\phi(x_s)$ is below the expandable row. Since $x_u$ is below $x_s$, we have~$\phi(x_u) \in B$. Moreover, $\phi(x_u)$ is to the left and above $\phi(b) \in H_u$, so $\phi(x_u) \in H_u^\rB$.
		
		Since $x_u \lth x_{u-1} \lth x_{u+1} \lth b$, we have $\phi(x_{u-1}), \phi(x_{u+1}) \in H_u$. Note that $\phi$ maps nothing to $H_u^\rT$, so $\phi(x_{u-1}) \in H_u^\rM \cup H_u^\rB$, and $\phi(x_{u+1})$ is below the expandable row, as $x_{u+1}$ is below $x_{u-1}$.
		
		Further, $x_u \ltv r \ltv b$, so $\phi(r) \in H_u^\rB$. Thus, $\phi$ does not map any 1-entries to the rows between $M$ and $H_u^\rB$, so $\height_\phi(M \cup L_{u+1}^\rB \cup R_{u+1}^\rB) \le \vd(x_{u+1}, b)$. Since $\phi(x_{u+1})$ is not in the top row of $M$, this means that $\phi(b)$ is below $L_{u+1}^\rB \cup R_{u+1}^\rB$, so $\phi(b) \in R_u^\rB$.
		
		Consider now $\phi(x_{u+2})$. Since $\phi(x_{u+1})$ is below the expandable row and, by \cref{p:box-sizes,p:box-sizes-special}, $\height_\phi(L_{u+1}^\rM \cup L_{u+1}^\rB) \le \vd(x_{u+1}, x_{u+2})$, we know that $\phi(x_{u+2})$ is below $L_{u+1}^\rB$. Further, $x_{u+2}$ is to the left of $b$, so $\phi(x_{u+2}) \in R_u^\rB$. Since $r$ is below $x_{u+2}$, this implies $\phi(r) \in R_u^\rB$.
		
		This means that $\phi(x_{u+1}) \in L_{u+1} \cup R_u$. Since $\height(R_u^\rB) < \vd(x_u, b) < \vd(x_{u+1}, b)$, we have $\phi(x_{u+1})$ above $R_u^\rB$, so $\phi(x_{u+1}) \in L_{u+1}^\rB \cup L_{u+1}^\rM \cup R_u^\rM$. We distinguish two cases:
		
		\begin{casedist}
			\item $\phi(x_{u+1}) \in L_{u+1}^\rB \cup L_{u+1}^\rM$. Since $\phi(x_{u+1})$ is below the expandable row, tallness of $X$ implies that $\phi$ maps no 1-entry to $R_u^\rM$. But then $\phi$ maps all 1-entries to $R_u^\rB$ or above and to the left of $R_u^\rB$ (recall that $\phi$ maps no 1-entries to $H_u^\rT$). Thus, $P$ is decomposable (since $\phi(b) \in R_u^\rB$, $\phi(t) \notin R_u^\rB$), a contradiction.
			
			\item $\phi(x_{u+1}) \in R_u^\rM$. Since $\phi(x_{u-1}) \in H_u^\rM \cup H_u^\rB$ is above $\phi(x_{u+1})$, this means \linebreak that~$\phi(x_{u-1}) \in L_{u+1}^\rM \cup R_u^\rM$. Note that $\phi$ cannot map any further 1-entries to $M$.
			But this means that $\phi$ maps all 1-entries either to $H_u^\rM \cup H_u^\rB$ or above and to the left of~$H_u^\rM \cup H_u^\rB$, again contradicting that $P$ is indecomposable.\qedhere
		\end{casedist}
	\end{proof}
	
	\begin{corollary}\label{p:t_Gs_b_Hs}
		There is some odd $s$ with $5 \le s \le m-5$ such that $\phi(t) \in G_s$ and $\phi(b) \in H_s$.
	\end{corollary}
	\begin{proof}
		Suppose first that $\phi(t) \in R_3$. Then \cref{p:t-not-borders,p:t-not-R3-or-b} imply \linebreak that $\phi(b) \in H_s$ for some odd $s$ with $5 \le s \le m-5$. But then \cref{p:t_Gs_b_right} implies \linebreak that $\phi(t) \notin R_3$, a contradiction. A similar argument shows that $\phi(b) \notin L_{m-2}$.
		
		As such, there are odd $s, u$ with $5 \le s, u \le m-5$ such that $\phi(t) \in G_s$ and $\phi(b) \in H_u$. \Cref{p:t_Gs_b_Hu} implies $u \le s$. If $u < s$, then $\phi(t), \phi(b) \in L_s \cup R_{s-1}$, contradicting \cref{p:not-t-L_s-b-R_s-1}. Thus, $s = u$.
	\end{proof}

	\subsubsection{\texorpdfstring{$\phi(t)$}{phi(t)}, \texorpdfstring{$\phi(b)$}{phi(b)} in the middle and close to each other}\label{sec:middle_samegroup}
	
	In this subsection, we show that \cref{p:t_Gs_b_Hs} also leads to a contradiction, which shows that our assumption that $S(P,X)$ contains $P$ must have been false. \Cref{fig:big-constr-middle-sketch} will be useful throughout this subsection.
	We start with the case $\phi(t) \in L_s$.
	
	\begin{lemma}\label{p:not-L_s}
		$\phi(t) \notin L_s$ and $\phi(b) \notin R_{s+1}$ for each odd $s$ with $5 \le s \le m-5$.
	\end{lemma}
	\begin{proof}
		By symmetry, it suffices to prove the first statement. Suppose $\phi(t) \in L_s$.
		\Cref{p:x_s_above_implies} implies that $\phi(x_s)$ is below the expandable row and thus to the right of $L_s$.
		
		We will consider several possibilities for the location of $\phi(b)$ and $\phi(r)$. Before, we make some observations. \Cref{p:t_Gs_b_Hs} implies that $\phi(b) \in H_s$. Since $\phi(x_s)$ is below the expandable \linebreak row, $\phi(x_{s+2}), \phi(b) \in B$, and thus $\phi(x_{s+2}), \phi(b) \in H_s^\rB = L_{s+1}^\rB \cup R_s^\rB \cup R_{s+1}^\rB$. Since \linebreak $x_{s+2} \lev x_{m-3} \ltv r$, this also implies $\phi(r) \in H_s^\rB$. This means that $\phi$ does not map any 1-entry into the rows between $M$ and $L_{s+1}^\rB \cup R_{s+1}^\rB$, so $\height_\phi(M \cup L_{s+1}^\rB) \le \vd(x_{s+1}, x_{s+2})$ and $\height_\phi(M \cup R_{s+1}^\rB) \le \vd(x_{s+1}, b)$.
		\begin{casedist}
			\item $\phi(b) \in L_{s+1}^\rB$. Then we have $\phi(r) \in L_{s+1}^\rB \cup R_{s+1}^\rB$. Since \linebreak $\height_\phi(M \cup L_{s+1}^\rB) \le \vd(x_{s+1}, x_{s+2}) < \vd(x_{s+1}, b)$, we have $\phi(x_{s+1}) \in T$. \linebreak Since~$x_{s+1}$ is to the left of~$b$, we have $\phi(x_{s+1}) \in L_s^\rT$. Moreover, $t \ltv \ell \ltv x_{s+1}$ implies $\phi(\ell) \in L_s$. But $\width(L_s) = \hd(\ell, x_s) - 1 < \hd(\ell, x_{s+1})$, a contradiction.
			
			\item $\phi(b) \in R_{s+1}^\rB$. Then $\phi(r) \in R_{s+1}^\rB$. Since $\height_\phi(M \cup R_{s+1}^\rB) \le \vd(x_{s+1}, b)$, we know that $\phi(x_{s+1})$ is above the expandable row, and therefore to the left of $R_{s+1}$ (by \cref{p:empty-boxes}).
			
			Since $x_{s-1}$ is above $x_{s+1}$, we have $\phi(x_{s-1}) \in T$, implying $\phi(x_{s-1}) \in L_s^\rT \cup R_s^\rT$ and thus $\phi(\ell) \in L_s$. Further, $\width(L_s^\rT) < \hd(\ell, x_{s-1})$, so $\phi(x_{s-1}) \in R_s^\rT$.
			
			Finally, since $\phi(x_{s+2}) \in B$ and $x_{s+2}$ is to the left of $x_{s+1}$, we have $\phi(x_{s+2}) \in L_{s+1}^\rB$. But then $\phi(x_{s+2})$ is to the left of $\phi(x_{s-1}) \in R_s^\rT$, while $x_{s+2}$ is to the right of $x_{s-1}$, a contradiction.
			
			\item $\phi(b), \phi(r) \in R_s^\rB$. We consider the location of $\phi(x_{s-1})$. Note that $\phi(x_{s-1}) \in G_s$, \linebreak since $t \lth x_{s-1} \lth r$.
			
			First, suppose that $\phi(x_{s-1}) \in R_s$. Let $q_s$ be the 1-entry of $P$ in the row below $x_s$. We have $\phi(q_s) \in B$, because $\phi(x_s)$ is below the expandable row. Since $X$ is tall, $q_s$ is to the right of $x_{s-1}$, so $\phi(q_s) \in R_s^\rB$. But $\height(R_s^\rB) \le \vd(x_s, b)-2 = \vd(q_s, b) - 1$, a contradiction.
			
			Second, suppose $\phi(x_{s-1}) \in L_{s+1}$. Since $\phi(t) \in L_s$, this means $\phi(x_{s-1}) \in M \cup B$. By tallness of $X$, there are no 1-entries in $P$ that are above and to the right of $x_{s-1}$, so $\phi$ does not map any 1-entry to $R_s^\rT$. Note that $\phi$ must map some 1-entry $y$ to $R_s^\rM$. Otherwise, $\phi$ maps all 1-entries to $R_s^\rB$ or above and to the left of $R_s^\rB$ (and $\phi(t) \notin R_s^\rB$, $\phi(b) \in R_s^\rB$), so $P$ is decomposable.
			
			By tallness of $X$, and since $y$ is to the left of $x_{s-1}$, we know that $y$ must be below~$x_{s-1}$. Thus, $\phi(x_{s-1})$ is in the top row of $M$, and $\phi(y)$ is in the bottom row of $M$. But since~$M$ only consists of two rows, $\phi$ maps no further 1-entries to $M$, so $\phi$ maps all 1-entries either to $H_s^\rM \cup H_s^\rB$ or to the left and above $H_s^\rM \cup H_s^\rB$. This again implies that $P$ is decomposable, a contradiction.
			
			Third, suppose $\phi(x_{s-1}) \in R_{s-1}^\rM$. Then $\phi(x_{s-1})$ is below the expandable row (by \cref{p:empty-boxes}), so $\phi(x_s) \in B$.
			But $x_s$ also lies to the left of $x_{s-1}$ and above $b$, a contradiction.
			
			Finally, suppose $\phi(x_{s-1}) \in L_s$. Since $t \lth x_s \lth x_{s-1}$, this implies $\phi(x_s) \in L_s$. But~$\phi(x_s)$ is below the expandable row, contradicting \cref{p:empty-boxes}.
			
			\item $\phi(b) \in R_s^\rB$ and $\phi(r) \in R_{s+1}^\rB$. Then $\phi(x_{s+2})$ is above and not to the right of $R_s^\rB$. Together with the fact that $\phi(x_{s+2}) \in B$, this implies $\phi(x_{s+2}) \in L_{s+1}^\rB$.
		
			Since $\height_\phi(L_{s+1}^\rB \cup M) \le \vd(x_{s+1}, x_{s+2})$, we know that $\phi(x_{s+1})$ is above the expandable row, and thus $\phi(x_{s-1}) \in T$, implying $\phi(x_{s-1}) \in L_s^\rT \cup R_s^\rT$.
			Moreover, since $\width(L_s^\rT) < \hd(\ell, x_{s-1})$, we have $\phi(x_{s-1}) \in R_s^\rT$. Now $\phi(x_{s-1}) \in R_s^\rT$ is to the right of $\phi(x_{s+2}) \in L_{s+1}^\rB$, but $x_{s-1}$ is to the left of $x_{s+2}$, a contradiction.\qedhere
		\end{casedist}
	\end{proof}
	
	The next few lemmas deal with the case that $\phi(t) \in L_{s+1}$.
	
	\begin{lemma}\label{p:not_t-b-L_s+1}
		Let $s$ be odd with $5 \le s \le m-5$. If $\phi(t) \in L_{s+1}$, then $\phi(b) \notin L_{s+1}$.
	\end{lemma}
	\begin{proof}
		Suppose $\phi(t), \phi(b) \in L_{s+1}$. Since $\width(L_{s+1}) < \hd(\ell, x_{s+1}) < \hd(\ell, b)$, we know that~$\phi(\ell)$ is to the left of $L_{s+1}$.
		
		$t \lth x_4 \lth x_s \lth x_{s+1} \lth b$ implies that $\phi(x_4), \phi(x_s) \phi(x_{s+1}) \in L_{s+1}$. Moreover, $\phi(x_4)$ is below $L_{s+1}^\rT$, since $x_4$ is below $\ell$. This implies that $x_{s+1}$ is below the expandable row, which in turn implies that $x_s \in L_{s+1}^\rB$.
		
		Since $r$ is below $x_s$ and above $b$, we have $\phi(r) \in L_{s+1}^\rB \cup R_{s+1}^\rB$, implying that $\phi$ maps no 1-entry into the rows between $M$ and $L_{s+1}^\rB \cup R_{s+1}^\rB$. But then $\height_\phi(L_{s+1}^\rB \cup M) \le \vd(x_{s+1}, b)$, so $\phi(x_{s+1})$ is above the expandable row, a contradiction.
	\end{proof}
	
	\begin{lemma}\label{p:L_s+1_b_R_s}
		Let $s$ be odd with $5 \le s \le m-5$. If $\phi(t) \in L_{s+1}$, then $\phi(b) \in R_s$.
	\end{lemma}
	\begin{proof}
		Assume $\phi(t) \in L_{s+1}$. By \Cref{p:t_Gs_b_Hs}, we have $\phi(b) \in H_s$. \Cref{p:not-L_s,p:not_t-b-L_s+1} imply that $\phi(b) \notin L_{s+1} \cup R_{s+1}$. If $\phi(b) \in L_{s+2}$, then $\phi(b)$ is above the expandable row. But then $\phi(r)$ is to the right of $R_s$, below $L_{s+2}^\rT$, and in $T$, which is impossible. The only remaining possibility is that $\phi(b) \in R_s$.
	\end{proof}
	
	\begin{lemma}\label{p:L_s+1_helper}
		Let $s$ be odd with $5 \le s \le m-5$. If $\phi(t) \in L_{s+1}$, then $\phi(x_{s-1}) \in L_{s+1}$ and~$\phi(x_{s+2}) \in R_s$.
	\end{lemma}
	\begin{proof}
		By \cref{p:L_s+1_b_R_s} and symmetry, it suffices to show that $\phi(x_{s-1}) \in L_{s+1}$. Suppose not. By \cref{p:L_s+1_b_R_s}, $\phi(b) \in R_s$, so $t \lth x_{s-1} \lth b$ implies that $\phi(x_{s-1}) \in R_s$. Let $q_s \in E(P)$ be the 1-entry of $P$ in the row directly below $x_s$.
		
		We claim that $\phi(x_s)$ is below the expandable row, and thus $\phi(q_s) \in B$.
		If ${\phi(x_{s-1}) \in M \cup B}$, then $\phi(x_s)$ is indeed below the expandable row, since $x_s$ is below $x_{s-1}$. Otherwise, ${\phi(x_{s-1}) \! \in \! R_s^\rT}$, which implies that $\phi(\ell) \in L_s^\rT \cup L_{s+1}^\rT$, so $\phi$ maps no 1-entry into the rows between $L_s^\rT \cup R_s^\rT$ and~$M$. Thus, $\height_\phi(R_s^\rT \cup R_s^\rM) \le \vd(x_{s-1}, x_s)$, implying that $\phi(x_s)$ is below the expandable row. This proves the claim.
		
		We have height $\height(R_s^\rB) \le \vd(x_s, b)-2 = \vd(q_s,b) - 1$, which implies \linebreak that $\phi(q_s) \notin R_s^\rB$. Since $X$ is tall, $q_s$ is to the right of $x_{s-1}$ and thus $\phi(q_s)$ is not to the left of $R_s$. Since ${\phi(q_s) \in B \setminus R_s^\rB}$, this implies that $\phi(q_s)$ is to the right of $R_s$, so $\phi(r)$ is to the right of $R_s$, and thus above $R_s^\rB$.
		
		Consider now $x_{s+2}$. First, $x_{s-1} \lth x_{s+2} \lth b$ implies that $\phi(x_{s+2}) \in R_s$. Since $x_s$ is below the expandable row, $\phi(x_{s+2}) \in R_s^\rB$. But then $\phi(x_{s+2})$ is below $\phi(r)$, a contradiction.
	\end{proof}
	
	\begin{lemma}\label{p:L_s+1_helper2}
		Let $s$ be odd with $5 \le s \le m-5$. If $\phi(t) \in L_{s+1}$, then $\phi(E(P)) \subseteq L_{s+1} \cup R_s$.
	\end{lemma}
	\begin{proof}
		We show that $\phi(\ell) \in L_{s+1}$ and $\phi(r) \in R_s$. By \cref{p:L_s+1_b_R_s}, we have $\phi(b) \in R_s$. Thus, by symmetry, it suffices to prove $\phi(\ell) \in L_{s+1}$. Suppose $\phi(\ell) \notin L_{s+1}$. Then $\phi(\ell)$ is below $L_{s+1}^\rT$. Since $x_{s-1}$ is below $\ell$ and $\phi(x_{s-1}) \in L_{s+1}$ by \cref{p:L_s+1_helper}, we have $\phi(x_{s-1}) \in L_{s+1}^\rM \cup L_{s+1}^\rB$.
		
		Since $\vd(x_{s-1}, x_{s+2}) > 1$, this means $\phi(x_{s+2}) \in B$. More precisely, by \cref{p:L_s+1_helper}, we have $\phi(x_{s+2}) \in R_s^\rB$. Since $r$ is below $x_{s+2}$, we also have $\phi(r) \in R_s^\rB$.
		
		Now consider $\phi(x_{s+1})$. Since $\height(R_s^\rB) < \vd(x_{s+1}, b)$, we know that $\phi(x_{s+1})$ is above $R_s^\rB$. Further, $x_{s+2} \lth x_{s+1} \lth b$ implies $\phi(x_{s+1}) \in R_s^\rT \cup R_s^\rM$.
		
		$x_{s-1}$ is above $x_{s+1}$, so we have $\phi(x_{s-1}), \phi(x_{s+1}) \in M$. This implies that $\vd(x_{s-1},x_{s+1}) \le \vd_\phi(\phi(x_{s-1}), \phi(x_{s+1})) = 1$, so $x_{s-1}$ is in the row directly above $x_{s+1}$ in $P$. Note that this means that the top row of $L_{s+1}^\rM$ contains precisely $x_{s-1}^{s+1}$, and thus $\phi(x_{s-1}) = x_{s-1}^{s+1}$.
		
		Finally, consider $\phi(x_s)$. We know that $\phi(x_s)$ is not to the right of $x_s^{s+1} \in L_{s+1}$, otherwise $\hd(\phi(x_s),\phi(x_{s-1})) < \hd(x_s^{s+1}, x_{s-1}^{s+1}) = \hd(x_s, x_{s-1})$. Moreover, $\phi(x_s)$ must be below $x_s^{s+1}$. Indeed, $\phi(r) \in R_s$ implies that $\phi$ maps no 1-entries between $L_{s+1}^\rB \cup R_{s+1}^\rB$ and $M$, implying $\vd_\phi(x_{s-1}^{s+1}, x_s^{s+1}) \le \vd(x_{s-1}, x_s) - 1$. So $\phi(x_s)$ is below and to the left of $x_s^{s+1}$. But then tallness of $X$ implies that $\phi(x_s)$ is below $\phi(b) \in H_s^\rB$, a contradiction.
	\end{proof}
	
	\begin{lemma}\label{p:not-L_s+1}
		$\phi(t) \notin L_{s+1}$ and $\phi(b) \notin R_s$ for each odd $s$ with $5 \le s \le m-5$.
	\end{lemma}
	\begin{proof}
		By symmetry, it suffices to show the first statement. Suppose $\phi(t) \in L_{s+1}$. \linebreak By \cref{p:L_s+1_helper,p:L_s+1_helper2}, we have $\phi(x_{s-1}) \in L_{s+1}$ as well as $\phi(x_{s+2}), \phi(b) \in R_s$ \linebreak and $\phi(P) \subseteq L_{s+1} \cup R_s$.
		
		Since $t \lth x_s \lth x_{s-1}$, we also have $\phi(x_s) \in L_{s+1}$. Symmetrically, $\phi(x_{s+1}) \in R_s$. \linebreak
		Let $p_{s+1} \in E(P)$ be the 1-entry of $P$ in the row directly above $x_{s+1}$, and let $q_s \in E(P)$ be the 1-entry in the row directly below $x_s$. Since $\height(L_{s+1}^\rT) = \vd(t,p_{s+1})-1$, we know that~$\phi(p_{s+1})$ is below $L_{s+1}^\rT$, and, symmetrically, $\phi(q_s)$ is above $R_s^\rB$. With $p_{s+1} \ltv x_{s+1} \ltv x_s \ltv q_s$, we have~$\phi(x_s), \phi(x_{s+1}), \phi(q_s), \phi(p_{s+1}) \in L_{s+1}^\rM \cup L_{s+1}^\rB \cup R_s^\rT \cup R_s^\rM$.
		
		Our goal for the remainder of the proof is to find two 1-entries $y_1, y_2 \in E(P)$ such that the sequence $Y = x_1, x_2, \dots, x_s, y_1, y_2, x_{s+1}, \dots, x_m$ is a traversal of $P$.
		For this, we have to show that (i) $y_2 \lth y_1$, as well as (ii) $y_1 \ltv x_{s+1}$, and (iii) $x_s \ltv y_2$ (note that then ${x_{s-1} \lth y_2 \lth y_1 \lth x_{s+2}}$ and ${x_{s-1} \ltv y_1 \ltv y_2 \ltv x_{s+2}}$ follow from tallness of $X$). The existence of such a traversal implies that $X$ is extendable, contradicting our assumption.
		
		We consider two cases. First, assume that $q_s$ is to the left of $p_{s+1}$. Then we simply \linebreak choose~$y_1 = p_{s+1}$ and $y_2 = q_s$. By definition, $p_{s+1}$ is above $x_{s+1}$ and $q_s$ is below $x_s$, and (i) follows by assumption.
		
		Second, assume that $p_{s+1}$ is to the left of $q_s$. Then either $\phi(p_{s+1}) \in L_{s+1}$ \linebreak or $\phi(q_s) \in R_s$. By symmetry, we can assume the former, which implies $\phi(p_{s+1}) \in L_{s+1}^\rM \cup L_{s+1}^\rB$. Since ${\phi(x_{s+1}) \in R_s^\rT \cup R_s^\rM}$ and $x_{s+1}$ is below $p_{s+1}$, we have $\phi(p_{s+1}), \phi(x_{s+1}) \in M$. More precisely, $\phi(p_{s+1}) = p_{s+1}^{s+1} \in L_{s+1}^\rM$ and $\phi(x_{s+1}) = q_s^s \in R_s^\rM$.
		
		Now choose $y_1, y_2 \in E(P)$ such that $y_1^{s+1} = \phi(x_{s-1})$ and $y_2^{s+1} = \phi(x_s)$. Note that this is well-defined, since $\phi(x_{s-1}), \phi(x_s) \in L_{s+1}$. We immediately have (i) from $x_s \lth x_{s-1}$.
		
		We now show (ii), that $y_1$ is above $x_{s+1}$.
		Since $x_{s-1}$ is above $x_{s+1}$ and $\phi(x_{s+1}) = q_s^s$, we know that $\phi(x_{s-1}) = y_1^{s+1}$ is above the expandable row. Since the expandable row in $L_{s+1}$ corresponds to the row containing $x_{s+1}$ in $P$, this means that $y_1$ is above $x_{s+1}$.
		
		Finally, we show (iii), that $y_2$ is below $x_s$.
		Since $\phi(r) \in R_s$, we know that $\phi$ maps no 1-entry into the rows between $M$ and $L_{s+1}^\rB$. Since $\phi$ also maps no 1-entry into the expandable row, we have $\vd_\phi(p_{s+1}^{s+1}, x_s^{s+1}) \le \vd(p_{s+1}, x_s) - 1$. As $\phi(p_{s+1}) = p_{s+1}^{s+1}$, this means that $\phi(x_s) = y_1^{s+1}$ is below $x_s^{s+1}$, implying (iii).
	\end{proof}
	
	The remaining cases are now easy:
	\begin{lemma}\label{p:not-remainder}
		$\phi(t) \notin R_{s-1} \cup R_s$ and $\phi(b) \notin L_{s+1} \cup L_{s+2}$ for each odd $s$ with $5 \le s \le m-5$.
	\end{lemma}
	\begin{proof}
		By symmetry, it suffices to show the first statement. Suppose $\phi(t) \in R_{s-1} \cup R_s$. By \Cref{p:t_Gs_b_Hs,p:not-L_s,p:not-L_s+1}, we have $\phi(b) = L_{s+1} \cup L_{s+2}$.
		
		Suppose first that $\phi(t) \in R_{s-1}$. Then $\phi(t)$ is below the expandable row, meaning that $\phi(\ell)$ is below $M$, but not to the right of $R_{s-1}$. But $\phi(b)$ is above $R_{s-1}^\rB$, implying that $\phi(\ell)$ is also above~$R_{s-1}^\rB$, a contradiction.
		
		Second, if $\phi(t) \in R_s$, then $\phi(b) \in L_{s+2}$, since $b$ is to the right of $t$. A symmetric argument shows that $\phi(r)$ is above $M$, below $L_{s+2}^\rT$, and not to the right of $L_{s+2}$, a contradiction.
	\end{proof}
	
	\Cref{p:not-L_s,p:not-L_s+1,p:not-remainder} imply that $\phi(t) \notin G_s$, contradicting \cref{p:t_Gs_b_Hs}. As such, our assumption that $\phi$ is an embedding of $P$ into $S(P,X)$ must be false. This concludes the proof of \cref{p:big-S-avoids}.
	
	\section{Conclusion and open problems}

	We showed that each indecomposable permutation matrix has bounded saturation function, thereby completing the classification of saturation functions of permutation matrices. Our proofs imply the upper bound $\sat(P,n) \le 9k^4$ for an indecomposable $k \times k$ permutation matrix $P$ (note that the largest witness $S(P,X)$ is not larger than $2k^2 \times k^2$, and \cref{p:vert_hor_wit} combines it with its 90-degree rotation, resulting in a $3k^2 \times 3k^2$ matrix). It would be interesting to improve this bound, especially if a simpler construction for patterns satisfying the conditions of \cref{p:big-S-avoids} can be found. Note that for general patterns with bounded saturation functions, no upper bound for $\sat(P,n)$ in terms of $P$ is known, as noted by Fulek and Keszegh~\cite{FulekKeszegh2021}.
	
	We also characterized a large class of non-permutation patterns with bounded saturation function, including very dense matrices (\cref{p:4-trav-gen}). Still, a full characterization of the saturation functions of all matrices remains out of reach. Note that there are indecomposable patterns without spanning oscillations, see, e.g., \cref{fig:indec-no-trav}. Thus, new techniques are likely required to fully resolve this problem.
	
	\begin{figure}
		\begin{align*}
			\begin{smallbulletmatrix}
				\o&  &  &  &  \\
				\o&  &  &  &  \\
				\o&  &  &  &  \\
				\o&  &  &  &  \\
				\o&\o&\o&\o&\o
			\end{smallbulletmatrix}
		\end{align*}
		\caption{An indecomposable non-permutation matrix without a spanning oscillation.}\label{fig:indec-no-trav}
	\end{figure}

	Our results trivially imply that every permutation matrix with a vertical witness also has a horizontal witness. It would be interesting to determine whether this is true for arbitrary patterns.
	
	It is also possible to consider the saturation functions of \emph{sets} of patterns. If $\fP$ is a set of patterns, let a matrix $M$ be $\fP$-saturating if $M$ avoids each $P \in \fP$, and adding a single 1-entry in $M$ creates an occurrence of some $P \in \fP$ in $M$. Let $\sat(\fP, n)$ be the minimum weight of $\fP$-saturating matrices. Since our witnesses for $k \times k$ permutation matrices have size at most $3k^2 \times 3k^2$, they avoid all patterns with one side of side length more than $3k^2$. Thus, if $\fP$ contains one indecomposable permutation matrix, and arbitrarily many much larger patterns, our results imply that $\sat(\fP,n) \in \fO(1)$.
	
	It would be interesting to determine the saturation functions for, say, all pairs of two permutation matrices of the same size. Gerbner, Nagy, Patkós, and Vizer~\cite{GerbnerEtAl2022} observed that certain saturation problems for two-dimensional posets can be reduced to saturation problems for sets of matrix patterns. However, these sets usually contain both permutation matrices and non-permutation matrices (of similar size).

	\appendix
	\section{Proof of \texorpdfstring{\cref{p:equiv-containment}}{Lemma \ref{p:equiv-containment}}}\label{app:proof-equiv-containment}
	
	\restateEquivContainment*
	\begin{proof}
		Say $P$ is $q \times s$ and $M$ is $m \times n$. Suppose $P = (p_{i,j})_{i,j}$ is contained in $M$. Then there are rows $r_1 < r_2 < \dots < r_q$ and columns $c_1 < c_2 < \dots < c_r$ such that $p_{i,j} \le m_{r_i,c_j}$ for each~$i \in [q], j \in [r]$. Now simply define $\phi(i,j) = (r_i, c_j)$. Clearly, $\phi(E(P)) \subseteq E(M)$. Moreover, consider $(i,j), (i',j') \in E(P)$. We have $i < i'$ if and only if $r_i < r_{i'}$, and $j < j'$ if and only if $r_j < r_{j'}$. Thus $\phi$ is an embedding of $P$ into $M$.
		
		Now suppose $\phi \colon E(P) \rightarrow E(M)$ is an embedding of $P$ into $M$. Note that $x, y \in E(P)$ are in the same row (resp. column) if and only if $\phi(x), \phi(y)$ are in the same row (resp. column). \linebreak Thus, $\phi(E(P))$ intersects exactly $q$ rows and $s$ columns. Let $r_1 < r_2 < \dots < r_q$ be those rows and $c_1 < c_2 < \dots < c_r$ be those columns. We show that $\phi(i,j) = (r_i, c_j)$ for \linebreak each $(i,j) \in E(P)$. Let $x_1, x_2, \dots, x_m \in E(P)$ such that $x_i$ is in the $i$-th row for \linebreak each $i \in [m]$, and let $r_i'$ be the row of $M$ containing $\phi(x_i)$. Clearly $r_1' \ge r_1$. By induction, we further have $r'_i \ge r_i$ for each $i \in [m]$. Similarly, $r_m' \le r_m$, and, again by induction, $r'_i \le r_i$ for each $i \in [m]$. This implies that $\phi(i,j)$ is in the $r_i$-th row of $M$ for every $(i,j) \in E(P)$. An analogous argument shows that $\phi(i,j)$ is in the $c_j$-th column of $M$.
		
		Since $\phi$ is an embedding, we have $(r_i, c_j) = \phi(i,j) \in E(M)$ for each $(i,j) \in E(P)$. \linebreak Thus, $p_{i,j} \le m_{r_i, c_j}$ for each $(i,j) \in [q] \times [s]$, so $P$ is contained in $M$.
	\end{proof}
	
	%%%%%%%%%%%%%%%%%%%%%%%%%%%%%%%%%%%%%%%%%%%%%%%%%%%
	%%%%%%%%%%%%%%%%%%%%%%%%%%%%%%%%%%%%%%%%%%%%%%%%%%%
	
	% ACKNOWLEDGEMENTS
	% Include acknowledgements to colleagues and referee here.
	% Funding and grant support should appear in footnotes on the front page, using the 
	% thanks command in the authors command (see above).
	
	\section*{Acknowledgements}
	
	The author would like to thank Jesse Geneson, L\'aszl\'o Kozma, and the anonymous reviewers for their helpful comments.
	
	%%%%%%%%%%%%%%%%%%%%%%%%%%%%%%%%%%%%%%%%%%%%%%%%%%%
	%%%%%%%%%%%%%%%%%%%%%%%%%%%%%%%%%%%%%%%%%%%%%%%%%%%
	
	% BIBLIOGRAPHY
	% Pease provide us with a bibtex file for your bibliography
	% You can find examples in ct-sample.bib. 
	% Please use the correct entrytype 
	%     article: any article published in a periodical like a journal article or magazine article
	%     book: a book
	%     booklet: like a book but without a designated publisher
	%     conference: a conference paper
	%     inbook: a section or chapter in a book
	%     incollection: an article in a collection
	%     inproceedings: a conference paper (same as the conference entry type)
	%     manual: a technical manual
	%     masterthesis: a Masters thesis
	%     misc: used if nothing else fits
	%     phdthesis: a PhD thesis
	%     proceedings: the whole conference proceedings
	%     techreport: a technical report, government report or white paper
	%     unpublished: a work that has not yet been officially published
	%
	% When it exists, please add a DOI to your reference using the doi field, this will be shown as a link in the final pdf
	% Similarly for preprints, please put the arXiv reference into the eprint field 
	% When necessary, you can provide a link through the URL field. Please do not give both the URL and the DOI.
	%
	% We encourage you to use MathSciNet https://mathscinet.ams.org/mathscinet
	% or other equivalent bibliography databases to get full and correct bibliographic entries.
	
	\bibliographystyle{alphaurl}
	\bibliography{saturation2}
\end{document}